\documentclass[12pt, letterpaper, reqno]{amsart}
\usepackage[hyperindex,breaklinks]{hyperref}
\usepackage{amsmath}
\usepackage{amsfonts}
\usepackage{mathrsfs}
\usepackage{graphicx}
\usepackage{color}
\usepackage{slashed}
\usepackage{braket}
\usepackage{extarrows}
\usepackage{amsmath, bm}
\usepackage{amssymb}
\usepackage[margin=1in]{geometry}
\usepackage{bbm}
\usepackage{hyperref}

\newcommand{\Res}{{\mathrm{Res}}}

\definecolor{yellow}{rgb}{1,.7,0}

\definecolor{pkured}{rgb}{0.55,0,0}

\newcommand{\be}{\begin{equation*}}
	\newcommand{\ee}{\end{equation*}}
\newcommand{\beq}{\begin{equation}}
	\newcommand{\eeq}{\end{equation}}

\numberwithin{equation}{section}

\newtheorem{cor}{Corollary}[section]

\newtheorem{lem}[cor]{Lemma}
\newtheorem{prop}[cor]{Proposition}
\newtheorem{thm}[cor]{Theorem}

\newtheorem{ex}[cor]{Example}
\newtheorem{rmk}[cor]{Remark}

\numberwithin{figure}{section}
\usepackage{ytableau}

\usepackage{tikz}
\newcounter{x}
\newcounter{y}
\newcounter{z}

\author{Chenglang Yang}
\email{yangcl@pku.edu.cn}
\address{Institute for Math and AI, Wuhan University, Wuhan 430072, China}
\address{Hua Loo-Keng Center for Mathematical Sciences \&
	Academy of Mathematics and Systems Science,
	Chinese Academy of Sciences,
	Beijing 100190, China}

\title[{The structures of simple and monotone Hurwitz numbers with varying genus}]
{The structures of simple Hurwitz numbers and monotone Hurwitz numbers with varying genus}

\begin{document}
\maketitle

\begin{abstract}
	We study the structures of ordinary simple Hurwitz numbers and monotone Hurwitz numbers with varying genus.
	More precisely,
	we prove that when the ramification type is fixed and the genus is treated as a variable,
	the connected monotone Hurwitz number is a linear combination of products of exponentials and polynomials,
	and the ordinary simple Hurwitz number is a linear combination of exponentials.
	Using these structural properties,
	we also derive the large genus asymptotics of these two kinds of Hurwitz numbers.
	As a result,
	we prove one conjecture,
	and disprove another,
	both proposed by Do, He and Robertson.
\end{abstract}

\setcounter{section}{0}
\setcounter{tocdepth}{2}


\section{Introduction}

The Hurwitz numbers were introduced by Hurwitz in the 1890s \cite{Hurw91}.
These numbers count branched covers between Riemann surfaces with prescribed ramification types over some branch points.
Hurwitz first studied branched covers with only simple ramifications and discovered rich structures for this kind of Hurwitz number \cite{Hurw91, Hurw01}.
In the past 30 years,
the connections between Hurwitz numbers and other fields of mathematics have been intensively studied
(see \cite{DYZ,ELSV,KL07,O00,OP06} and references therein).
Numerous significant structures of these numbers have been uncovered,
particularly in their recursion formulas,
polynomial structures,
connections to integrable hierarchies and moduli spaces of Riemann surfaces.

In this paper,
we use the notation $H_{g;\mu}$ to represent the connected simple Hurwitz number labeled by genus $g$ and ramification type $\mu$,
which counts branched covers between a genus $g$ Riemann surface and the Riemann sphere with a branch point whose ramification type is $\mu$ and some other branch points have simple ramifications (see subsection \ref{sec:subsec hurw} for a detailed description of this notation).
There are many generalizations and variants of the simple Hurwitz number (see, for examples, \cite{ACEH,DK16,GGN13,GGN14,HOY,JPT,KLS,KPS,O00}),
including orbifold Hurwitz number and monotone Hurwitz number.
The latter is also one of the objects studied in this paper.
More precisely,
we use the notation $\vec{H}_{g;\mu}$ to represent the connected monotone Hurwitz number labeled by $g$ and $\mu$
(see subsection \ref{sec:subsec mono} for a precise definition).
As the main results of this paper,
we prove that,
for any fixed partition $\mu$ and considering genus $g$ as a variable,
the connected monotone Hurwitz number $\vec{H}_{g;\mu}$ is a linear combination of products of exponentials and polynomials,
and the connected simple Hurwitz number $H_{g;\mu}$ is a linear combination of exponentials.
These structures also give the large genus asymptotics of these two kinds of Hurwitz numbers.
As a result,
we prove one conjecture,
and disprove another conjecture,
both of which are proposed by Do, He and Robertson \cite{DHR}.

The monotone Hurwitz numbers were introduced by Goulden, Guay-Paquet and Novak when they studied the Harish-Chandra--Itzykson--Zuber matrix integral \cite{GGN14},
and play an important role in the asymptotics of matrix integral.
This kind of Hurwitz number is intensively studied and possesses many interesting structures, including topological recursion and polynomiality (see, for examples \cite{DDM,DK16,DKPS,GGN13,GGN16,KLS,KPS}).
Recently,
Do, He, and Robertson \cite{DHR} proposed a conjecture about the monotone Hurwitz numbers.
They expected that,
for a fixed partition $\mu$,
the connected monotone Hurwitz number $\vec{H}_{g;\mu}$,
as a function of $g$,
is a linear combination of exponentials plus a linear term (see Conjecture 4.7 in \cite{DHR} for more details).
Their prediction was inspired by a similar structure for the ordinary simple Hurwitz numbers and their data for monotone Hurwitz numbers with $|\mu|\leq 5$.
However,
as we will demonstrate,
their prediction does not hold for general $\mu$.
The first main result of this paper is the following structural theorem for the connected monotone Hurwitz numbers with varying genus.
\begin{thm}
	\label{thm:main thm for mono}
	For the connected monotone Hurwitz numbers,
	we fix the degree $d\geq2$ and the ramification type $\mu=(\mu_1,...,\mu_l)$, which is a partition of $d$.
	Then for arbitrary genus $g$,
	we have
	\begin{align}
		\mu_1\cdots\mu_l\cdot\vec{H}_{g;\mu}
		=\sum_{k=1}^{d-1}
		\sum_{1\leq i\leq \lfloor (d-1)/k\rfloor}
		\Big(\vec{C}(\mu;k,i)\cdot (2g-2+d+l)^{i-1} \cdot k^{2g-2+d+l}\Big),
	\end{align}
	where the coefficients $\vec{C}(\mu;k,i)\in\mathbb{Q}$ with $1\leq k\leq d-1, 1\leq i\leq \lfloor (d-1)/k\rfloor$ are independent of the genus $g$.
\end{thm}
We list some examples of connected monotone Hurwitz numbers in subsection \ref{sec:mono pf}.
The simplest example that contradicts the Do--He--Robertson's prediction is $\mu=(3,3)$,
which is discussed in Example \ref{ex:disprove dhr}.

The asymptotics of geometric invariants attracted much attention from mathematicians in recent years.
For examples,
large genus behavior of intersection numbers over moduli spaces was studied in \cite{Agg21,DGZZ,EGGGL,GNYZ,GY24,LX14,MZ15} and references therein.
The large genus and large degree asymptotics of classical single Hurwitz numbers have also been investigated in \cite{DYZ}.
It is remarkable that the structural property established in Theorem \ref{thm:main thm for mono} is well-suited to this problem (see also \cite{DHR}) and directly provides a form of large genus asymptotics for the connected monotone Hurwitz numbers.
More precisely,
we obtain the following
\begin{prop}\label{prop:main prop for mono}
	When fixing the degree $d\geq2$ and the ramification type $\mu=(\mu_1,...,\mu_l)$, which is a partition of $d$,
	the connected monotone Hurwitz numbers $\vec{H}_{g,\mu}$ have the following large genus asymptotics,
	\begin{align}
		\mu_1\cdots\mu_l\cdot\vec{H}_{g;\mu}
		=\frac{2\cdot (d-1)^{d-2}}{d!\cdot (d-2)!} \cdot (d-1)^{2g-2+d+l}
		+O\big((d-2)^{2g-2+d+l}\big)
	\end{align}
	as $g\rightarrow\infty$.
\end{prop}

For the ordinary simple Hurwitz number,
there is a similar result.
First,
for the classical single Hurwitz number $H_{g;(1^d)}$ case,
where $(1^d)=(1,...,1)$ is a partition of $d$,
Hurwitz discovered this structural property \cite{Hurw91, Hurw01}.
He proved that $H_{g;(1^d)}$,
as a function of $g$,
is a linear combination of exponentials.
Later,
other mathematicians further studied this structure.
For instance,
Dubrovin, Yang, and Zagier \cite{DYZ} revisited this structure and proved additional properties using the Toda equation satisfied by the generating function of classical single Hurwitz numbers.
Recently,
Do, He, and Robertson \cite{DHR} extended Hurwitz's result to the general $\mu$ case.
More precisely,
they proved that,
for any fixed $\mu$,
which is a partition of $d\geq2$,
there are some integers $C(\mu;k)\in\mathbb{Z}$ for $1\leq k\leq \binom{d}{2}$ such that,
\begin{align*}
	H_{g;\mu}
	=\frac{2}{d! \cdot \mu_1 \cdots \mu_{l}}
	\cdot \sum_{k=1}^{\binom{d}{2}}
	\Big(C(\mu;k) \cdot k^{d+2g-2+l}\Big)
\end{align*}
holds for all genus $g\in\mathbb{Z}_{\geq0}$.
Their method is based on the fermionic Fock space (see \cite{J15,O00}) and has the potential to be applied to study some other models.
They also obtained the leading term $C(\mu;\binom{d}{2})=1$ and proposed a conjecture which states the vanishing of $C(\mu;k)$ for $\binom{d-1}{2}<k<\binom{d}{2}$.
See their Conjecture 4.1 in \cite{DHR}.
The second main result of this paper is a proof of their conjecture.
Our method not only proves this vanishing result
but also provides a formula for the next non-zero term $C(\mu;\binom{d-1}{2})$.
This leads to a better asymptotic formula for the connected simple Hurwitz numbers.
The main result concerning simple Hurwitz numbers can be stated as the following
\begin{thm}[Do--He--Robertson's Theorem 1.5 and Conjecture 4.1 \cite{DHR}]
	\label{thm:main thm for hurw}
	For the connected simple Hurwitz numbers,
	we fix the degree $d\geq2$ and the ramification type $\mu=(\mu_1,...,\mu_{l})$,
	which is a partition of $d$.
	Then there are some integers $C(\mu;k)\in\mathbb{Z}, 1\leq k\leq \binom{d}{2}$ such that,
	for arbitrary genus $g\geq0$,
	\begin{align}
		H_{g;\mu}
		=\frac{2}{d! \cdot \mu_1 \cdots \mu_{l}}
		\cdot \sum_{k=1}^{\binom{d}{2}}
		\Big(C(\mu;k) \cdot k^{d+2g-2+l}\Big).
	\end{align}
	Moreover,
	these coefficients satisfy
	$C(\mu;k)=0$ for $\binom{d-1}{2}<k<\binom{d}{2}$
	and
	\begin{align}
		C(\mu;\binom{d}{2})=1,
		\qquad \qquad C(\mu;\binom{d-1}{2})=-d\cdot \#\{i\ |\ \mu_i=1\}.
	\end{align}
	Thus,
	for fixed $\mu$,
	as $g\rightarrow\infty$,
	we have
	\begin{align}\label{eqn:H asym}
		\begin{split}
		H_{g;\mu}
		=\frac{2}{d!\cdot \mu_1\cdots \mu_l}
		\binom{d}{2}^{d+2g-2+l}
		-&\frac{2\cdot \#\{i\ |\ \mu_i=1\}}{(d-1)!\cdot \mu_1\cdots \mu_l}\binom{d-1}{2}^{d+2g-2+l}\\
		&\qquad+O\Bigg(\bigg(\binom{d-1}{2}-1\bigg)^{d+2g-2+l}\Bigg).
		\end{split}
	\end{align}
\end{thm}

It was pointed out in Section 5.1 of \cite{DYZ} and Section 4.1 of \cite{DHR} that the approximation of the first term in the asymptotic formula \eqref{eqn:H asym} is already extremely precise.
It seems that,
based on our Theorem \ref{thm:main thm for hurw},
one possible reason for this precise is the vanishing of the next $(d-2)$ terms.

The remainder of this paper is organized as follows.
In Section \ref{sec:pre},
we first review the definition of Hurwitz numbers and their relation to counting constellations.
We then review the KP hierarchy and a formula for computing connected $n$-point function of KP tau-functions,
which serves as the primary tool used in this paper to study the structure of Hurwitz numbers.
In Section \ref{sec:mono},
we investigate the structure of connected monotone Hurwitz numbers,
which proves Theorem \ref{thm:main thm for mono} and Proposition \ref{prop:main prop for mono}.
In Section \ref{sec:Hurw},
we study the large genus asymptotics of connected simple Hurwitz numbers and prove Theorem \ref{thm:main thm for hurw}.

\section{Preliminaries}
\label{sec:pre}
In this section,
we review the notions of Schur polynomials,
Hurwitz numbers,
and KP hierarchy.
The fascinating thing is that there are deep relations between them.

\subsection{Partitions and Schur polynomials}
In this subsection,
we review the integer partitions and Schur polynomials.
They play an important role in combinatorics and mathematical physics.
We recommend the book \cite{Mac}.

Let $d$ be a non-negative integer.
A partition of $d$ is a sequence of positive integers $\mu=(\mu_1,...,\mu_l)$ satisfying
\begin{align*}
	\mu_1\geq\mu_2\geq\cdots\geq\mu_l,
	\qquad \sum_{i=1}^l \mu_i=d.
\end{align*}
We use $|\mu|=d$ and $l(\mu)=l$ to denote the size and length of the corresponding partition $\mu$, respectively.
For convenience,
the empty partition $\emptyset=(0)$ is a partition of $0$ and its length is also $0$.
We denote $\mathcal{P}$ as the set of all partitions.
Each partition $\mu$ has a unique Young diagram corresponding to it,
which has $\mu_i$ boxes in its $i$-th row.
Then,
intuitively,
the conjugate partition $\mu^t$ of $\mu$ is obtained by reflecting the Young diagram corresponding to $\mu$ along the main diagonal.
More precisely,
$\mu^t$ is a partition of length $\mu_1$ and
\begin{align*}
	\mu^t_j:=\#\{i|\mu_i\geq j\},
	\qquad 1\leq j\leq \mu_1.
\end{align*}

The Frobenius notation of a partition $\mu$ is given by $(m_1,...,m_r|n_1,...,n_r)$,
where $r$ is the maximal integer such that $\mu_r\geq r$ and
\begin{align*}
	m_i:=\mu_i-i,
	\qquad n_i:=\mu^t_i-i,
	\qquad 1\leq i \leq r.
\end{align*} 
As a result,
the notation $(a|b)$ represents $(a+1,1,,,1)$, which is a partition of $a+b+1$.
For a partition $\mu=(\mu_1,...,\mu_l)$,
we define
\begin{align}\label{eqn:def kappa}
	\kappa_{\mu}=\sum_{i=1}^l \mu_i(\mu_i-2i+1).
\end{align}
This combinatorial number is related to the special character value of irreducible representations of symmetric groups and is important in studying the integrability of Hurwitz numbers \cite{O00}.

Schur introduced the Schur polynomials when he studied the representation theory of symmetric groups $S_d$.
They are a family of polynomials in variables $\mathbf{t}=(t_1,t_2,...)$ labeled by partitions.
More precisely,
for the case $l(\mu)\leq1$,
the elementary Schur polynomials are defined by
\begin{align*}
	\sum_{n=0}^\infty \big(z^n \cdot s_{(n)}(\mathbf{t})\big)
	:=\exp\Big(\sum_{m=1}^\infty \big(t_m \cdot z^m\big)\Big).
\end{align*}
In general,
the Schur polynomial labeled by $\mu=(\mu_1,...,\mu_l)$ is
\begin{align*}
	s_{\mu}(\mathbf{t})
	=\det\big(s_{(\mu_i+j-i)}(\mathbf{t})\big)_{1\leq i,j\leq l}.
\end{align*}
Then it is obvious that if we assign $\deg t_n=n$,
then $s_{\mu}(\mathbf{t})$ is a homogeneous polynomial of degree $|\mu|$.

It is well known that the set of irreducible representations of $S_{d}$ has a one-to-one correspondence to the set of partitions of $d$.
More precisely,
for a partition $\mu$ of $d$,
one can simply expand the Schur polynomial $s_{\mu}(\mathbf{t})$ as
\begin{align*}
	s_{\mu}(\mathbf{t})
	=\sum_{\lambda\in\mathcal{P}}
	\bigg(\frac{\chi^{\mu}_{\lambda}}{z_{\lambda}} \cdot \prod_{i=1}^{l(\lambda)}(\lambda_i\cdot t_{\lambda_i})\bigg),
\end{align*}
where $z_{\lambda}:=\prod_{j=1}^\infty m_j(\lambda)!\cdot j^{m_j(\lambda)}$ and $m_j(\lambda):=\#\{i|\lambda_i=j\}$.
These coefficients $\chi^{\mu}_{\lambda}$ are non-zero only when $|\lambda|=d$ and are actually character values of the irreducible representation labeled by $\mu$ of symmetric group $S_{d}$.
For example,
the dimension of the irreducible representation labeled by $\mu$ is given by the hook-length formula
\begin{align*}
	\chi^{\mu}_{(1^d)}
	=\frac{d!}{\prod_{(i,j)\in\mu} h(i,j)}, 
\end{align*}
where the notation $(i,j)\in\mu$ means $1\leq i\leq l(\mu), 1\leq j\leq \mu_i$ and $h(i,j):=\mu_i+\mu^t_j-i-j+1$.
This dimension formula corresponds to (see I. 3. Example 5 in \cite{Mac})
\begin{align}\label{eqn:schur hook}
	s_{\mu}(\delta_{k,1})
	=[t_1^d]\ s_{\mu}(\mathbf{t})
	=\frac{1}{\prod_{(i,j)\in\mu} h(i,j)},
\end{align}
where the notation $[t_1^d]\ s_{\mu}(\mathbf{t})$ means taking the coefficient of $t_1^d$ in $s_{\mu}(\mathbf{t})$.
This formula will be useful in deriving the affine coordinates of generating functions of simple Hurwitz numbers and monotone Hurwitz numbers in the next two sections.

\subsection{The single Hurwitz numbers}
\label{sec:subsec hurw}
In this subsection,
we review the geometric definition of Hurwitz numbers and their equivalent definition by counting constellations.

The ordinary Hurwitz number counts branched covers between Riemann surfaces with fixed ramification types over branch points.
In this paper,
we focus on the case that counts branched covers over the Riemann sphere $\mathbb{P}^1$.
This case is particularly interesting and has been intensively studied in recent years
(see \cite{ACEH,DYZ,ELSV,GJV,HOY,KL07,O00,OP06} and references therein).

More precisely,
let $d$ be a positive integer and $\mu^1,...,\mu^k$ be $k$ partitions of $d$.
We choose $k$ points $q_1,...,q_k\in\mathbb{P}^1$.
Without loss of generality,
we can assume $q_1=\infty\in\mathbb{P}^1$.
For a branched cover $f$ from a possibly disconnected genus $g$ Riemann surface $\Sigma_g$ to $\mathbb{P}^1$,
if $q\in\mathbb{P}^1$ has $l(\mu)$ pre-images such that, near the $j$-th pre-image $f$ looks like $z\mapsto z^{\mu_j}, j=1,...,l(\mu)$ locally,
we say that the ramification type of $f$ over the point $q$ is $\mu$.
Then the ramification types of $f$ over the whole Riemann sphere are called $(\mu^1,...,\mu^k)$ if $f$ ramifies over these $k$ points $q_1,...,q_k$ with ramification types described by partitions $\mu^1,...,\mu^k$, respectively,
and is unramified over other points.
For such a branched cover $f$,
an automorphism $\phi$ of $f$ is a bi-holomorphic map from $\Sigma_g$ to itself such that $f=f\circ\phi$ and we extra assume $\phi$ preserves the pre-images of $q_1=\infty\in\mathbb{P}^1$.

\begin{rmk}\label{rmk:explain autmu}
	Actually,
	the condition $\phi$ preserves the pre-images of $q_1$ only causes an extra factor 
	$|{\rm Aut}(\mu)|:=\prod_{j=1}^\infty m_j(\mu)!$ if $m_j(\mu):=\#\{i|\mu_i=j\}$ to the Hurwitz numbers,
	where $\mu$ is the ramification type of corresponding branched cover over $q_1$.
	We use this notation following \cite{DHR,GJV} for convenience.
\end{rmk}

The (disconnected) genus $g$ Hurwitz number labeled by partitions $\mu^1,...,\mu^k$ is then defined by
\begin{align*}
	H^{\bullet}_{g}(\mu^1,...,\mu^k)
	:=\sum_{f} \frac{1}{|\text{Aut}(f)|},
\end{align*}
where the sum is over branched covers $f$ from a possibly disconnected genus $g$ Riemann surface to $\mathbb{P}^1$ with ramification types described by $\mu^1,...,\mu^k$.
The Riemann--Hurwitz formula tells us $H^{\bullet}_{g}(\mu^1,...,\mu^k)$ is zero unless the following equation holds
\begin{align}\label{eqn:RH formula}
	2g-2
	=d\cdot (k-2)-\sum_{i=1}^k l(\mu^i).
\end{align}
The connected Hurwitz number $H_{g}(\mu^1,...,\mu^k)$ is defined similarly and only counts branched covers from a connected Riemann surface to $\mathbb{P}^1$ as
\begin{align*}
	H_{g}(\mu^1,...,\mu^k)
	:=\sum_{f:\text{connected}} \frac{1}{|\text{Aut}(f)|}.
\end{align*}
In this paper,
we are interested in the so-called simple Hurwitz numbers.
In this special case,
most of the branch points have simple ramification type described by $(2,1^{d-2})$ and only the point $q_1=\infty\in\mathbb{P}^1$ has an arbitrary ramification type.
We use the following notation to represent the simple Hurwitz number labeled by genus $g$ and partition $\mu$,
\begin{align*}
	H_{g;\mu}
	:=H_{g}\big(\mu,(2,1^{d-1})...,(2,1^{d-1})\big),
\end{align*}
where there are $b$ branch points with simple ramification type,
and $b$ is equal to $2g-2+d+l(\mu)$ by the Riemann--Hurwitz formula \eqref{eqn:RH formula}.
Similarly,
one can also define the disconnected simple Hurwitz numbers $H^{\bullet}_{g;\mu}$.

Below,
we introduce the equivalent definition of simple Hurwitz numbers via counting constellations,
which provides a combinatorial method to compute Hurwitz numbers.
We recommend the books \cite{CM16,LZ04} for more explanations.
A constellation is a sequence of permutations $(\sigma_1,...,\sigma_n)$,
where $\sigma_i\in S_d, 1\leq i\leq n$ such that the following two conditions are satisfied:

i). The subgroup $\langle\sigma_1,...,\sigma_n\rangle$ acts transitively on the set $\{1,...,d\}$,

ii). The product of $\sigma_i$ is equal to the identity permutation,
i.e., $\sigma_1\cdots\sigma_n=id$.

\noindent The constellations are related to the geometry of moduli spaces, Grothendieck's dessins d'Enfants, matrix integrals, and integrable hierarchies.
See \cite{LZ04} for more details on these connections.

Via counting constellations,
the connected simple Hurwitz number $H_{g;\mu}$ is equal to 
\begin{align}\label{eqn:eq def H}
	\begin{split}
	H_{g;\mu}
	=\frac{|{\rm Aut}(\mu)|} {|\mu|!}
	\cdot\#\{(\sigma_1,...,\sigma_{b+1})
	|&\ \sigma_1\in C_{\mu},
	\ \sigma_i\in C_{(2,1^{d-2})}, 2\leq i\leq b+1,\\
	&\ \sigma_1\cdots\sigma_{b+1}=id,
	\ \langle\sigma_1,...,\sigma_{b+1}\rangle\text{\ is\ transitive}\},
	\end{split}
\end{align}
where $C_{\mu}$ is the conjugate class of $S_d$ represented by $\mu$ such that permutations in $C_{\mu}$ have cycles of length $\mu_1,...,\mu_{l(\mu)}$.
See Section 7 in \cite{CM16} for a detailed construction of the one-to-one correspondence between branched covers and constellations.
The factor $|{\rm Aut}(\mu)|$ comes from our definition of Hurwitz numbers.
See the explanation in Remark \ref{rmk:explain autmu}.
One can find the equivalent definition of the disconnected simple Hurwitz numbers via counting constellations similarly with dropping out the transitive condition.

\subsection{The monotone Hurwitz numbers}
\label{sec:subsec mono}
In this subsection,
we review the monotone Hurwitz numbers and mainly follow the papers \cite{GGN13,GGN14}.

Goulden, Guay-Paquet, and Novak introduced the monotone Hurwitz numbers when they studied the asymptotic expansion of the Harish-Chandra--Itzykson--Zuber matrix integral.
These numbers count restricted subsets of branched covers,
and can also be viewed as counting random monotone walks on symmetric groups.
We use the following definition via counting constellations
\begin{align}\label{eqn:def mono}
	\begin{split}
		\vec{H}_{g;\mu}
		:=&\frac{|{\rm Aut}(\mu)|} {|\mu|!}
		\cdot\#\{(\sigma_1,...,\sigma_{b+1})
		|\ \sigma_1\in C_{\mu},
		\ \sigma_i\in C_{(2,1^{d-2})}, 2\leq i\leq b+1,\\
		&\qquad\ \sigma_1\cdots\sigma_{b+1}=id,
		\ \langle\sigma_1,...,\sigma_{b+1}\rangle\text{\ is\ transitive},
		\ (\sigma_2,...,\sigma_{b+1}) \text{\ is\ monotone}\},
	\end{split}
\end{align}
where most of the notations are similar to those used in equation \eqref{eqn:eq def H},
and the last condition means that,
if we write $\sigma_i=(A_i,B_i), 2\leq i\leq b+1$ such that $A_i<B_i$,
then we must have $B_2\leq B_3\leq\cdots\leq B_{b+1}$.

The generating function of monotone Hurwitz numbers satisfies the monotone cut-and-join equation as in Theorem 1.1 of \cite{GGN13} (See also \cite{DKPS}).
This cut-and-join equation provides an efficient algorithm to compute monotone Hurwitz numbers.
Moreover,
this equation is used in studying the polynomiality of monotone Hurwitz numbers, KP integrability, and topological recursion (see \cite{DDM,DKPS,GGN13}).
These numbers can also be generalized to the mixed Hurwitz numbers as in \cite{GGN16}.

\subsection{Connected $n$-point function of tau-functions of the KP hierarchy}
\label{sec:KP}
In this subsection,
we review the KP hierarchy and a formula computing the connected $n$-point function of KP tau-functions.

The Kadomtsev--Petviashvili (KP) hierarchy is an infinite-dimensional generalization of the KP equation (see \cite{DJM} for an introduction).
It consists of infinitely many nonlinear partial differential equations satisfied by a tau-function $\tau(\mathbf{t})$, where $\mathbf{t}=(t_1,t_2,...)$ are time variables.
This hierarchy is considered as an infinite-dimensional integrable system by many mathematicians and it has deep connections with Kac--Moody Lie algebra and mathematical physical model.
In this paper,
we are only concerned with the formal power series tau-function.
That is to say,
we always assume $\tau(\mathbf{t})\in\mathbb{C}[\![\mathbf{t}]\!]$.
An equivalent description of the KP hierarchy is the following Hirota bilinear relation satisfied by $\tau(\mathbf{t})$,
\begin{align}
	\label{eqn:hirota}
	\int\ e^{\xi(\mathbf{t}-\mathbf{t}',z)}
	\tau(\mathbf{t}-[z^{-1}]) \tau(\mathbf{t}'+[z^{-1}])\ dz=0,
\end{align}
where $\xi(\mathbf{t},z):=\sum_{k=1}^\infty t_k z^k$, and
\begin{align*}
	\mathbf{t}\pm[z^{-1}]
	=(t_1\pm z^{-1}, t_2\pm\frac{1}{2}z^{-2},
	...,t_n\pm\frac{1}{n}z^{-n},...).
\end{align*}

The set of all non-zero formal power series tau-functions of the KP hierarchy (scaling by a non-zero constant) is isomorphic to a semi-infinite dimensional Grassmannian,
which is called the Sato Grassmannian in literature (see \cite{Sato}).
There is a natural cell decomposition for this Grassmannian,
and tau-functions corresponding to points in the big cell satisfy $\tau(\mathbf{0})\neq0$,
which forms an affine space $\mathbb{C}^{\mathbb{N}\times\mathbb{N}}$.
Without loss of generality,
when discussing this kind of tau-function,
we always assume $\tau(\mathbf{0})=1$.
Then, each of these kinds of tau-functions can be described by the so-called affine coordinates $\{a_{n,m}\}_{n,m\in\mathbb{Z}_{\geq0}}$,
which are just the coordinates of the point, corresponding to this tau-function, in the affine space (see also \cite{BY17, Z15}).
These numbers $\{a_{n,m}\}_{n,m\in\mathbb{Z}_{\geq0}}$ are canonical coordinates of this tau-function.

We denote $\tau^A(\mathbf{t})$ the tau-function of KP hierarchy, which is determined by these affine coordinates $\{a_{n,m}\}_{n,m\in\mathbb{Z}_{\geq0}}$ and satisfies $\tau^A(\mathbf{0})=1$.
To obtain a detailed description of this tau-function,
we consider the following Schur polynomial expansion of this tau-function,
\begin{align}\label{eqn: tau^A}
	\tau^A(\mathbf{t})
	=\sum_{\lambda\in\mathcal{P}}
	\big(c_\lambda \cdot s_{\lambda}(\mathbf{t})\big)
\end{align}
with $c_{\emptyset}=1$.
It is known that the Hirota bilinear relation \eqref{eqn:hirota} is equivalent to the following relations for the coefficients 
\begin{align*}
	c_{\lambda}
	=\det \big(c_{(m_i|n_j)}\big)_{1\leq i,j\leq r}
\end{align*}
where $(m_1,...,m_r|n_1,...,n_r)$ is the Frobenius notation of $\lambda$.
Thus the set of all coefficients $\{c_{\lambda}\}_{\lambda\in\mathcal{P}}$ is determined by $\{c_{(a|b)}\}_{a,b\geq0}$.
Moreover,
for the tau-function $\tau^A(\mathbf{t})$ specified by these affine coordinates $\{a_{n,m}\}_{n,m\in\mathbb{Z}_{\geq0}}$,
these numbers $c_{(a|b)}$ are given by the affine coordinates up to a sign as (see Section 3 in \cite{Z15}),
\begin{align}\label{eqn:a as c}
	a_{n,m}=(-1)^n \cdot c_{(m|n)}.
\end{align}
Thus, equation \eqref{eqn: tau^A} gives us a precise description of the tau-function $\tau^A(\mathbf{t})$ since all $c_{\lambda}, \lambda\in\mathcal{P}$ are determined from affine coordinates $\{a_{n,m}\}_{n,m\in\mathbb{Z}_{\geq0}}$.

In principle,
from the above discussions,
all properties of the tau-function $\tau^A(\mathbf{t})$ are determined by its affine coordinates $\{a_{n,m}\}_{n,m\in\mathbb{Z}_{\geq0}}$.
In mathematical physical models,
the tau-function is in general an exponential-type generating function of geometric invariants. 
So, we are particularly interested in computing the expansion coefficients of $\log \tau(\mathbf{t})$.
We call them the connected $n$-point correlators and their generating function as connected $n$-point function.
There is a ``determinantal" type formula that computes the connected $n$-point function.
This kind of formula has appeared in many literatures in different forms.
See,
for examples, \cite{BE12,TW94,Z15} and references therein.
This paper follows the notation in \cite{Z15}.
We first need to consider the following generating function of the affine coordinates $\{a_{n,m}\}_{n,m\in\mathbb{Z}_{\geq0}}$
\begin{align}\label{eqn:A as a}
	A(z,w):=\sum_{n,m\geq0}
	\big(a_{n,m} \cdot z^{-n-1} w^{-m-1}\big).
\end{align}
This generating function is related to the fermionic two-point function in boson-fermion correspondence and the Baker--Akhiezer function in integrable system.
Then,
the formula for connected $n$-point function is,
for $n=1$ (see equation (211) in \cite{Z15}),
\begin{align}\label{eqn:conn n=1}
	\sum_{j\geq1}
	\bigg(\frac{\partial \log \tau^A(\mathbf{t})}{\partial t_{j}} \Big|_{\mathbf{t}=0}
	\cdot z^{-j-1}\bigg)
	=A(z,z),
\end{align}
and for $n\geq2$ (see Theorem 5.3 in \cite{Z15}),
\begin{align}\label{eqn:conn n}
	\begin{split}
	\sum_{j_1,\cdots,j_n\geq 1}&
	\bigg(\frac{\partial^n \log \tau^A(\mathbf{t})}{\partial t_{j_1} \cdots \partial t_{j_n}} \Big|_{\mathbf{t}=0}
	\cdot z_1^{-j_1-1} \cdots z_n^{-j_n-1}\bigg) \\
	=& (-1)^{n-1}\cdot \sum_{\text{$n$-cycles } \sigma}
	\Big( \prod_{i=1}^n \widehat A (z_{\sigma^{i}(1)},z_{\sigma^{i+1}(1)}) \Big)
	- i_{z_1,z_2}\frac{\delta_{n,2}}{(z_1-z_2)^2},
	\end{split}
\end{align}
where
\begin{align}\label{eq-AandAhat}
	\widehat A(z_i,z_j) =
	\begin{cases}
		i_{z_i,z_j} \frac{1}{z_i-z_j} + A(z_i,z_j),\qquad
		& i<j,\\
		i_{z_j,z_i} \frac{1}{z_i-z_j} + A(z_i,z_j), &i>j,
	\end{cases}
\end{align}
and the notation $i_{z,w} f(z,w)$ means we expand the function $f(z,w)$ in the region $|z|>|w|$,
i.e., for example, $i_{z,w} \frac{1}{z-w} = \sum_{h\geq 0} z^{-1-h}w^h$.

\section{The structure of monotone Hurwitz numbers}
\label{sec:mono}
In this section,
we study the structure of monotone Hurwitz numbers using the KP integrability of their generating function.
We prove Theorem \ref{thm:main thm for mono} and Proposition \ref{prop:main prop for mono}.

\subsection{The KP integrability of monotone Hurwitz numbers}
The monotone Hurwitz numbers are intensively studied (see \cite{DDM,DK16, DKPS,GGN13, GGN14, GGN16, KLS,KPS,WY23} and references therein),
including their connection to the matrix model,
integrable hierarchy,
combinatorial structure and topological recursion.
The generating function of monotone Hurwitz numbers is defined as
\begin{align}\label{eqn:def vectau}
	\vec{\tau}(\hbar;\mathbf{t})
	=\exp
	\bigg(\sum_{g\geq 0 \atop \mu\in\mathcal{P}}
	\frac{\hbar^{2g-2+l(\mu)+|\mu|}
		\cdot \vec{H}_{g;\mu}}
		{|{\rm Aut}(\mu)|}
	\cdot \prod_{i=1}^{l(\mu)} \big(\mu_i\cdot t_{\mu_i}\big)\bigg),
\end{align}
where $|{\rm Aut}(\mu)|=\prod_{j=1}^\infty m_j(\mu)!$ if $m_{j}(\mu):=\#\{i|\mu_i=j\}$.

\begin{rmk}
	Notice that the above definition of generating function $\vec{\tau}(\hbar;\mathbf{t})$ of monotone Hurwitz numbers matches with the generating function used by Goulden--Guay-Paquet--Novak in \cite{GGN13} (see their equation (1.1)).
	Actually,
	our monotone Hurwitz number $\vec{H}_{g;\mu}$ is equal to $\frac{|{\rm Aut}(\mu)|}{|\mu|!}$ times their monotone Hurwitz number $\vec{H}^{2g-2+l(\mu)+|\mu|}(\mu)$ from the definition \eqref{eqn:def mono} and our $t_i$ is equal to their $p_i/i$.
\end{rmk}

From the definition of monotone Hurwitz numbers via counting constellations reviewed in subsection \ref{sec:subsec mono},
this generating function $\vec{\tau}(\hbar;\mathbf{t})$ admits a Schur polynomial expansion as
(see Section 2 in \cite{GGN16})
\begin{align}\label{eqn:vectau as s}
	\vec{\tau}(\hbar;\mathbf{t})
	=\sum_{\lambda\in\mathcal{P}}
	\bigg(\prod_{(i,j)\in\lambda}\frac{1}{1-c(i,j)\cdot \hbar}
	\cdot s_{\lambda}(\delta_{k,1})
	\cdot s_{\lambda}(\mathbf{t})\bigg),
\end{align}
where $(i,j)\in\lambda$ means $1\leq i\leq l(\lambda), 1\leq j\leq \lambda_i$,
$c(i,j):=j-i$,
and $s_{\lambda}(\delta_{k,1})$ can be computed by formula \eqref{eqn:schur hook}.

\begin{rmk}
	In Section 2 of \cite{GGN16},
	Goulden, Guay-Paquet and Novak derived the Schur polynomial expansion of generating function of double mixed Hurwitz numbers using the representation theory of symmetric groups.
	The above equation \eqref{eqn:vectau as s} is a special case of theirs by setting $u=0$ and $p_k(A)=\delta_{k,1}$.
\end{rmk}

It is known that the generating function $\vec{\tau}(\hbar;\mathbf{t})$ of monotone Hurwitz numbers is a tau-function of the KP hierarchy,
which can also be derived from the Schur polynomial expansion \eqref{eqn:vectau as s}.
Moreover,
this expansion formula also gives the affine coordinates corresponding to this tau-function as
\begin{lem}\label{lem:affine vectau}
	The affine coordinates of the generating function of monotone Hurwitz numbers $\vec{\tau}(\hbar;\mathbf{t})$ are given by
	\begin{align*}
		\vec{a}_{n,m}
		=\frac{(-1)^n}{(m+n+1)\cdot m!\cdot n!}
		\cdot \prod_{j=-m}^n \frac{1}{1+j\hbar}.
	\end{align*}
\end{lem}
\begin{proof}
For any $m, n\in\mathbb{Z}_{\geq0}$,
the coefficient of $s_{(m|n)}(\mathbf{t})$ in the Schur polynomial expansion \eqref{eqn:vectau as s} is
\begin{align*}
	\prod_{(i,j)\in(m|n)}\frac{1}{1-c(i,j)\cdot \hbar}
	\cdot s_{(m|n)}(\delta_{k,1})
	=\prod_{j=-m}^n \frac{1}{1+j\hbar}
	\cdot \frac{1}{(m+n+1)\cdot m!\cdot n!}.
\end{align*}
Then, following equation \eqref{eqn:a as c},
the affine coordinate $\vec{a}_{n,m}$ of the tau-function $\vec{\tau}(\hbar;\mathbf{t})$ is the product of $(-1)^n$ and above equation.
\end{proof}

When studying the connected monotone Hurwitz numbers via the formulas \eqref{eqn:conn n=1} and \eqref{eqn:conn n},
we need the generating function $\vec{A}(z,w)$ of affine coordinates corresponding to $\vec{\tau}(\hbar;\mathbf{t})$,
which is defined in equation \eqref{eqn:A as a}.
In this case,
its explicit formula is
\begin{align}\label{eqn:A for mono}
	\vec{A}(z,w)
	=\sum_{n,m\geq0} \bigg(\frac{(-1)^n}{(m+n+1)\cdot m!\cdot n!}
	\cdot \prod_{j=-m}^n \frac{1}{1+j\hbar}
	\cdot z^{-n-1} w^{-m-1}\bigg).
\end{align}

\subsection{The structure of monotone Hurwitz numbers with varying genus}
\label{sec:mono pf}
In this subsection,
we use the KP integrability of $\vec{\tau}(\hbar;\mathbf{t})$ to study the structure of connected monotone Hurwitz numbers
and prove Theorem \ref{thm:main thm for mono}.

\begin{thm}[= Theorem \ref{thm:main thm for mono}]
	\label{thm:thm for mono}
	For the connected monotone Hurwitz numbers,
	we fix the degree $d\geq2$ and the ramification type $\mu=(\mu_1,...,\mu_l)$, which is a partition of $d$.
	Then for arbitrary genus $g$,
	we have
	\begin{align}
		\mu_1\cdots\mu_l\cdot\vec{H}_{g;\mu}
		=\sum_{k=1}^{d-1}
		\sum_{1\leq i\leq \lfloor (d-1)/k\rfloor}
		\Big(\vec{C}(\mu;k,i)\cdot (2g-2+d+l)^{i-1} \cdot k^{2g-2+d+l}\Big),
	\end{align}
	where the coefficients $\vec{C}(\mu;k,i)\in\mathbb{Q}$ with $1\leq k\leq d-1, 1\leq i\leq \lfloor (d-1)/k\rfloor$ are independent of the genus $g$.
\end{thm}
\begin{proof}
We first prove the $l(\mu)=1$ case where $\mu=(d)$.
To this end,
we consider the following generating function of connected monotone Hurwitz numbers labeled by fixed partition $\mu=(d)$ and arbitrary genus, 
\begin{align*}
	d\cdot \sum_{g=0}^\infty 
	\big(\hbar^{2g-1+d} \cdot \vec{H}_{g;(d)}\big)
	=&[z^{-d-1}]
	\ \sum_{i=1}^\infty
	\bigg(z^{-i-1}
	\cdot \frac{\partial \log\vec{\tau}(\hbar;\mathbf{t})}{\partial t_i} |_{\mathbf{t}=0}\bigg),
\end{align*}
where the tau-function $\vec{\tau}(\hbar;\mathbf{t})$ is defined in equation \eqref{eqn:def vectau},
and the notation $[z^{-d-1}]$ means we take the coefficient of $z^{-d-1}$ in corresponding expression.
Recall that the generating function of affine coordinates $\vec{A}(z,w)$ corresponding to $\vec{\tau}(\hbar;\mathbf{t})$ is explicitly given in equation \eqref{eqn:A for mono}.
Then, we can use the formula \eqref{eqn:conn n=1} of connected one-point function to compute the above equation.
The result is
\begin{align}\label{eqn:vecH1 as }
	\begin{split}
	d\cdot \sum_{g=0}^\infty 
	\big(\hbar^{2g-1+d} \cdot \vec{H}_{g;(d)}\big)
	=&[z^{-d-1}]\ A(z,z)\\
	=&\sum_{n,m\geq0 \atop n+m=d-1} 
	\bigg(\frac{(-1)^n}{(m+n+1)\cdot m!\cdot n!}
	\cdot \prod_{j=-m}^n \frac{1}{1+j\hbar}\bigg).
	\end{split}
\end{align}
One can notice that,
when considering as a function of $\hbar$,
the right-hand side of above equation is a rational function, and on the whole extended complex plane $\overline{\mathbb{C}}=\mathbb{C}\cup\{\infty\}$ it has only simple poles at $\{-\frac{1}{d-1}, -\frac{1}{d-2},..., -1, 1,...,\frac{1}{d-1}\}$.
Thus,
it can be decomposed into several simple fractions as
\begin{align}\label{eqn:vecH1 as D}
	d\cdot \sum_{g=0}^\infty 
	\big(\hbar^{2g-1+d} \cdot \vec{H}_{g;(d)}\big)
	=\sum_{k=-d+1}^{d-1}
	\frac{\vec{D}((d);k)}{1-k\hbar}.
\end{align}
for some coefficients $\vec{D}((d);k)\in\mathbb{Q}, -d+1\leq k\leq d-1$.
We now analyze the properties of the above equation.
First,
the left-hand side of equation \eqref{eqn:vecH1 as D} is an odd function if $d$ is even, and is an even function if $d$ is odd.
Thus,
we have
\begin{align*}
	\vec{D}((d);k)
	=(-1)^{d-1} \cdot \vec{D}((d);-k).
\end{align*}
Then,
equation \eqref{eqn:vecH1 as D} reduces to
\begin{align*}
	d\cdot \sum_{g=0}^\infty 
	\big(\hbar^{2g-1+d} \cdot \vec{H}_{g;(d)}\big)
	=\begin{cases}
		\vec{D}(\mu;0)
		+2\sum_{k=1}^{d-1}
		\frac{\vec{D}((d);k)}{1-k^2\hbar^2},
		& \text{\ if\ } d \text{\ is\ odd\ },\\
		2\sum_{k=1}^{d-1}
		\frac{k\hbar\cdot \vec{D}((d);k)}{1-k^2\hbar^2},
		& \text{\ if\ } d \text{\ is\ even\ }.
	\end{cases}
\end{align*}
By taking Taylor expansion of the right-hand side of above equation,
and taking the coefficient of $\hbar^{2g-1+d}$,
we obtain
\begin{align*}
	d\cdot \vec{H}_{g;(d)}
	=\sum_{k=1}^{d-1}
	\Big( \vec{C}((d);k) \cdot k^{2g-1+d}\Big)
\end{align*}
for $d\geq2$,
where the coefficients $\vec{C}((d);k)\in\mathbb{Q}, 1\leq k\leq d-1$ are given by
\begin{align}\label{eqn:vecC ad vecD 1}
	\vec{C}((d);k)
	=2\cdot \vec{D}((d);k).
\end{align}
For keeping compatible with notations in the general $l=l(\mu)\geq2$ case,
we denote $\vec{C}((d);k)$ by $\vec{C}((d);k,1)$.
We thus obtain the proof of the theorem for this $l=l(\mu)=1$ case.

Next,
we deal with the general $l=l(\mu)\geq2$ case.
We still consider the following generating function of connected monotone Hurwitz numbers labeled by fixed $\mu$ and arbitrary $g$,
\begin{align*}
	\mu_1\cdots\mu_l\cdot\sum_{g=0}^\infty
	&\big(\hbar^{2g-2+d+l} \cdot \vec{H}_{g;\mu}\big)\\
	&=[z_1^{-\mu_1-1}\cdots z_{l}^{-\mu_{l}-1}]
	\ \sum_{j_1,\cdots,j_l\geq 1}
	\bigg(\frac{\partial^l \log \vec{\tau}(\hbar;\mathbf{t})}{\partial t_{j_1} \cdots \partial t_{j_l}} \Big|_{\mathbf{t}=0}
	\cdot z_1^{-j_1-1} \cdots z_l^{-j_l-1}\bigg).
\end{align*}
In this case,
we need to use the formula \eqref{eqn:conn n} of connected $n$-point function to compute the right-hand side of above equation.
Then we have
\begin{align}\label{eqn:vecHl as AA}
	\begin{split}
	\mu_1\cdots\mu_l\cdot\sum_{g=0}^\infty
	&\big(\hbar^{2g-2+d+l} \cdot \vec{H}_{g;\mu}\big)\\
	&=(-1)^{l-1}
	\cdot [z_1^{-\mu_1-1}\cdots z_{l}^{-\mu_{l}-1}]
	\ \sum_{\text{$l$-cycles } \sigma}
	\Big( \prod_{i=1}^l \widehat{\vec{A}} (z_{\sigma^i(1),\sigma^{i+1}(1)}) \Big),
	\end{split}
\end{align}
where $\widehat{\vec{A}}(z_i,z_j)$ in this case is explicitly given by
\begin{align}\label{eqn:vecH hat{A} formula}
	\widehat{\vec{A}}(z_i,z_j) =
	\begin{cases}
		\sum\limits_{h\geq 0} z_i^{-1-h}z_j^h
		+\sum\limits_{n,m\geq0} \Big(\frac{(-1)^{n}\cdot \prod_{j=-m}^n \frac{1}{1+j\hbar}}{(m+n+1)\cdot m!\cdot n!}
		\cdot z_i^{-n-1} z_j^{-m-1}\Big),
		& \text{if}\ i<j,\\
		-\sum\limits_{h\geq 0} z_j^{-1-h}z_i^h
		+\sum\limits_{n,m\geq0} \Big(\frac{(-1)^{n}\cdot \prod_{j=-m}^n \frac{1}{1+j\hbar}}{(m+n+1)\cdot m!\cdot n!}
		\cdot z_i^{-n-1} z_j^{-m-1}\Big),
		& \text{if}\ i>j.
	\end{cases}
\end{align}
In the above equation,
we have used formula \eqref{eqn:A for mono} for the generating function of affine coordinates $\vec{A}(z,w)$ corresponding to the tau-function $\vec{\tau}(\hbar;\mathbf{t})$.

Similar to the $l=1$ case,
one can notice that,
the right-hand side of equation \eqref{eqn:vecHl as AA},
considered as a function of $\hbar$,
is a rational function and on the whole extended complex plane $\overline{\mathbb{C}}$ it can at most have poles at $\{-\frac{1}{d-1}, -\frac{1}{d-2},..., -1, 1,...,\frac{1}{d-1}\}$.
Moreover,
for each $k=\pm1,...,\pm (d-1)$,
if we denote $N_k\in\mathbb{Z}_{\geq0}$ the order of pole at $\hbar=\frac{1}{k}$ of this rational function,
since we need to take the coefficient of $z_1^{-\mu_1-1}\cdots z_{l}^{-\mu_{l}-1}$ in the right-hand side of equation \eqref{eqn:vecHl as AA},
we must have
\begin{align*}
	|k|\cdot N_k \leq (d-1), \text{\ for\ all\ } k=-d+1,...,-1,1,...,d-1.
\end{align*}
For example,
we have $N_{d-1}\leq1$.
Thus, the right-hand side of equation \eqref{eqn:vecHl as AA} can be decomposed to a linear combination of $\frac{1}{(1-k\hbar)^i}$ with $1\leq |k|\leq d-1, 1\leq i\leq N_k$ plus a constant term.
More precisely, there are some rational numbers $\vec{D}(\mu;0)$ and $\vec{D}(\mu;k,i)$ with $1\leq |k|\leq d-1, 1\leq i\leq N_k$ such that
\begin{align}\label{eqn:vecHl as D}
	\mu_1\cdots\mu_l\cdot \sum_{g=0}^\infty
	\big(\hbar^{2g-2+d+l} \cdot \vec{H}_{g;\mu}\big)
	=\vec{D}(\mu;0)
	+\sum_{-d+1\leq k\leq d-1, k\neq0 \atop 1\leq i\leq N_k} 
	\frac{\vec{D}(\mu;k,i)}{(1-k\hbar)^i}.
\end{align}
One can notice that the left-hand side of above equation is an odd function if $d+l$ is odd and is an even function if $d+l$ is even.
Thus,
we have
\begin{align*}
	\vec{D}(\mu;k,i)
	=(-1)^{d+l} \cdot \vec{D}(\mu;-k,i).
\end{align*}
By expanding the right-hand side of equation \eqref{eqn:vecHl as D} and taking the coefficient of $\hbar^{2g-2+d+l}$ with $d\geq2$,
we obtain,
\begin{align}\label{eqn:vecH last step}
	\mu_1\cdots\mu_l\cdot \vec{H}_{g;\mu}
	=\sum_{k=1}^{d-1}
	\sum_{1\leq i\leq \lfloor (d-1)/k\rfloor}
	\Big(\vec{C}(\mu;k,i)\cdot (2g-2+d+l)^{i-1} \cdot k^{2g-2+d+l}\Big).
\end{align}
In the above equation,
the coefficients $\vec{C}(\mu;k,i)$ are obtained from $\vec{D}(\mu;k,i)$ as,
for any $\mu$ and $k$,
\begin{align}\label{eqn:vecC as vecDl}
	\sum_{1\leq i\leq \lfloor (d-1)/k\rfloor}
	\Big(\vec{C}(\mu;k,i) \cdot x^{i-1}\Big)
	=2\sum_{1\leq i\leq \lfloor (d-1)/k\rfloor}
	\bigg(\vec{D}(\mu;k,i) \cdot \frac{\prod_{s=1}^{i-1}(x+s)}{(i-1)!}\bigg),
\end{align}
where $x$ is an indeterminate, and the above equation holds as an equality between two polynomials of $x$.
Thus,
equation \eqref{eqn:vecH last step} has proved the theorem for this $l=l(\mu)\geq2$ case.
\end{proof}

\begin{rmk}
	In the above proof,
	we denote $N_k$ the order of pole at $\hbar=\frac{1}{k}$ of the rational function in equation \eqref{eqn:vecHl as AA}.
	These numbers are meaningful. For example, if $i>N_k$,
	we have $\vec{C}(\mu;k,i)=0$.
	In the above theorem,
	we take a naive upper bound of these numbers $N_k\leq\frac{d-1}{|k|}$.
	There are some better upper bounds for these numbers $N_k$.
	For example,
	since
	in the right-hand side of equation \eqref{eqn:vecHl as AA} there are total of $l=l(\mu)$ functions $\hat{\vec{A}}(\cdot,\cdot)$,
	we must have $N_k\leq l$ for all $k$.
	This upper bound shows $\vec{C}(\mu;k,i)=0$ for all $i>l=l(\mu)$.
\end{rmk}

\begin{ex}
We provide two examples of the $l=l(\mu)=1$ case.
When $\mu=(5)$,
we have
\begin{align*}
	5\cdot \sum_{g=0}^\infty \big(\hbar^{2g+4} \cdot \vec{H}_{g;(5)}\big)
	=&\frac{14 \cdot \hbar^{4}}
	{\prod_{i=1}^4 (1-i^2 \cdot \hbar^2)}.
\end{align*}
Thus,
after taking the coefficient of $\hbar^{2g+4}$ in both sides of above equation,
we obtain
\begin{align*}
	5\cdot \vec{H}_{g;(5)}
	=\frac{8}{45}\cdot 4^{2g+4}
	-\frac{9}{20}\cdot 3^{2g+4}
	+\frac{14}{45}\cdot 2^{2g+4}
	-\frac{7}{180}\cdot 1^{2g+4}.
\end{align*}
This matches the data in Table 3 of \cite{DHR} up to a global constant $5$.

When $\mu=(10)$,
we have,
\begin{align*}
	10\cdot \sum_{g=0}^\infty \big(\hbar^{2g+9} \cdot \vec{H}_{g;(10)}\big)
	=&\frac{4862 \cdot \hbar^{9}}
	{\prod_{i=1}^9 (1-i^2 \cdot \hbar^2)},
\end{align*}
and
\begin{align*}
	10\cdot \vec{H}_{g;(10)}
	=&\frac{59049}{100352000}\cdot 9^{2g+9}
	-\frac{16384}{4465125}\cdot 8^{2g+9}
	+\frac{14000231}{1492992000}\cdot 7^{2g+9}\\
	&-\frac{153}{12250}\cdot 6^{2g+9}
	+\frac{1328125}{146313216}\cdot 5^{2g+9}
	-\frac{2176}{637875}\cdot 4^{2g+9}\\
	&+\frac{1989}{3584000}\cdot 3^{2g+9}
	-\frac{221}{8930250}\cdot 2^{2g+9}
	+\frac{2431}{36578304000}\cdot 1^{2g+9}.
\end{align*}
\end{ex}

\begin{ex}\label{ex:disprove dhr}
	We provide two examples of the $l=l(\mu)=2$ case.
	The first example that contradicts the prediction of Do--He--Robertson (Conjecture 4.7 in \cite{DHR}) is $\mu=(3,3)$.
	We have
	\begin{align*}
		9\cdot \sum_{g=0}^\infty
		\big(\hbar^{2g+6} \cdot \vec{H}_{g;(3,3)}\big)
		=\frac{60 \hbar^{6} \cdot (264 \hbar^{4}-65 \hbar^{2}+5)}{\prod_{i=3}^5 (1-i^2\cdot \hbar^2)
		\cdot \prod_{i=1}^2 (1-i^2\cdot \hbar^2)^2}.
	\end{align*}
	The above equation, as a function of $\hbar$, has four second order poles at $\hbar=\pm1, \pm \frac{1}{2}$.
	This phenomenon will not appear in the $l(\mu)=1$ or $|\mu|\leq 5$ cases.
	Actually,
	on the one hand,
	when $l(\mu)=1$, the generating function of connected monotone Hurwitz numbers has only simple poles.
	On the other hand, when $|\mu|\leq5$, the corresponding generating function has at most two second-order poles at $\hbar=\pm1$ and simple poles at other points
	(this is why Do, He and Robertson made the prediction as their Conjecture 4.7 in \cite{DHR}).
	After taking the coefficient of $\hbar^{2g+6}$ in both sides of above equation,
	we obtain
	\begin{align*}
		9\cdot \vec{H}_{g;(3,3)}
		=&\frac{125}{1728}\cdot 5^{2g+6}
		-\frac{32}{135}\cdot 4^{2g+6}
		+\frac{81}{320}\cdot 3^{2g+6}\\
		&-\frac{2}{9}\cdot (2g+6)\cdot 2^{2g+6}
		+\frac{92}{135} \cdot 2^{2g+6}
		-\frac{17}{72}\cdot (2g+6)\cdot 1^{2g+6} -\frac{1663}{2160} \cdot 1^{2g+6}.
	\end{align*}
	
	When $\mu=(5,3)$,
	we have
	\begin{align*}
		15\cdot \sum_{g=0}^\infty
		\big(\hbar^{2g+8} \cdot \vec{H}_{g;(5,3)}\big)
		=&\frac{315 \hbar^{8} \cdot (3932 \hbar^{4}-395 \hbar^{2}+15)}
		{\prod_{i=3}^7 (1-i^2\cdot \hbar^2)
			\cdot \prod_{i=1}^2 (1-i^2\cdot \hbar^2)^2},
	\end{align*}
	and
	\begin{align*}
		15 \cdot \vec{H}_{g;(5,3)}=&\frac{16807}{2073600}\cdot 7^{2g+8}
		-\frac{27}{700} \cdot 6^{2g+8}
		+\frac{40625}{580608}\cdot 5^{2g+8}
		-\frac{176}{2025}\cdot 4^{2g+8}
		+\frac{5373}{25600}\cdot 3^{2g+8}\\
		&-\frac{1}{10}\cdot(2g+8)\cdot2^{2g+8}
		-\frac{1013}{11340}\cdot 2^{2g+8}
		-\frac{37}{2880}\cdot(2g+8)\cdot1^{2g+8}
		- \frac{23593}{322560}\cdot1^{2g+8}.
	\end{align*}
\end{ex}

\begin{ex}
	We provide an example of the $l=l(\mu)=3$ case.
	When $\mu=(3,2,1)$,
	we have
	\begin{align*}
		6 \cdot \sum_{g=0}^\infty
		\big(\hbar^{2g+7} \cdot \vec{H}_{g;(3,2,1)}\big)
		=&\frac{240 \hbar^{7} \cdot (230 \hbar^{4}-71 \hbar^{2}+6)}
		{\prod_{i=3}^5 (1-i^2\cdot \hbar^2)
			\cdot \prod_{i=1}^2 (1-i^2\cdot \hbar^2)^2},
	\end{align*}
	and
	\begin{align*}
		6 \cdot \vec{H}_{g;(3,2,1)}
		=&\frac{125}{1728}\cdot 5^{2g+7}
		-\frac{8}{27}\cdot 4^{2g+7}
		+\frac{99}{320}\cdot 3^{2g+7}\\
		&-\frac{2}{9}\cdot (2g+7)\cdot 2^{2g+7}
		+\frac{182}{135} \cdot 2^{2g+7}
		-\frac{55}{72}\cdot (2g+7)\cdot 1^{2g+7} -\frac{43}{27} \cdot 1^{2g+7}.
	\end{align*}
\end{ex}

\begin{rmk}
	Our connected monotone Hurwitz number $\vec{H}_{g;\mu}$ is equal to $\frac{1}{\mu_1\cdots\mu_l}$ times Do--He--Robertson's $\vec{H}_{g;\mu}$ listed in Table 3 of \cite{DHR} from above examples and their data.
\end{rmk}

\subsection{The large genus asymptotics of monotone Hurwitz numbers}
In this subsection,
we derive a formula for the coefficient $\vec{C}(\mu;d-1,1)$,
which gives the large genus asymptotics of connected monotone Hurwitz numbers and proves Proposition \ref{prop:main prop for mono}.

\begin{prop}
	When fixing the degree $d\geq2$ and the ramification type $\mu=(\mu_1,...,\mu_l)$, which is a partition of $d$,
	the connected monotone Hurwitz numbers $\vec{H}_{g,\mu}$ have the following large genus asymptotics
	\begin{align*}
		\mu_1\cdots\mu_l\cdot\vec{H}_{g;\mu}
		=\frac{2\cdot (d-1)^{d-2}}{d!\cdot (d-2)!} \cdot (d-1)^{2g-2+d+l}
		+O\big((d-2)^{2g-2+d+l}\big)
	\end{align*}
	as $g\rightarrow\infty$.
\end{prop}
\begin{proof}
For convenience,
about the notation $\vec{C}(\mu;k,i)$ in Theorem \ref{thm:thm for mono},
if $i$ does not satisfy $1\leq i\leq \lfloor(d-1)/k\rfloor$,
we set $\vec{C}(\mu;k,i):=0$.
Then one can directly verify for all $d\geq2$ and any partition $\mu$ of $d$,
we have $\vec{C}(\mu;d-1,i)=\vec{C}(\mu;d-2,i)=0$ for $i\geq2$.
Thus,
from Theorem \ref{thm:thm for mono},
we have
\begin{align*}
	\mu_1\cdots\mu_l\cdot\vec{H}_{g;\mu}
	=\vec{C}(\mu;d-1,1) \cdot (d-1)^{2g-2+d+l}
	+O\big((d-2)^{2g-2+d+l}\big)
\end{align*}
as $g\rightarrow\infty$.
As a consequence,
to prove this proposition,
We only need to compute the coefficient $\vec{C}(\mu;d-1,1)$.

We first prove the $l=l(\mu)=1$ case.
Recall that for convenience,
we have used $\vec{C}((d);k,1)$ and $\vec{D}((d);k,1)$ to represent $\vec{C}((d);k)$ and $\vec{D}((d);k)$, respectively.
From equation \eqref{eqn:vecH1 as },
we have
\begin{align*}
	\vec{D}((d);d-1,1)
	=-(d-1)\cdot \Res_{\hbar=\frac{1}{d-1}}
	\bigg(\frac{1}{d!\cdot \prod_{j=-d+1}^0 (1+j\hbar)}\bigg)
	=\frac{(d-1)^{d-2}}{d!\cdot (d-2)!}.
\end{align*}
Thus, the coefficient $\vec{C}((d);d-1,1)$ can be obtained from equation \eqref{eqn:vecC ad vecD 1} as
\begin{align*}
	\vec{C}((d);d-1,1)
	=2\cdot \vec{D}((d);d-1,1)
	=\frac{2\cdot (d-1)^{d-2}}{d!\cdot (d-2)!},
\end{align*}
which proves the theorem in this case.

We then deal with the $l=l(\mu)\geq2$ case.
Recall that in the proof of Theorem \ref{thm:thm for mono},
we have proved the following formula for the generating function of connected monotone Hurwitz numbers,
\begin{align}\label{eqn:vecHl as AA2}
	\begin{split}
	\mu_1\cdots\mu_l\cdot\sum_{g=0}^\infty
	&\big(\hbar^{2g-2+d+l} \cdot \vec{H}_{g;\mu}\big)\\
	=&(-1)^{l-1}
	\cdot [z_1^{-\mu_1-1}\cdots z_{l}^{-\mu_{l}-1}]
	\ \sum_{\text{$l$-cycles } \sigma}
	\Big( \prod_{i=1}^l \widehat{\vec{A}} (z_{\sigma^i(1)},z_{\sigma^{i+1}(1)}) \Big),
	\end{split}
\end{align}
where $\widehat{\vec{A}}(z_i,z_j)$ is given by
\begin{align}\label{eqn:vecH hat{A} formula2}
	\widehat{\vec{A}}(z_i,z_j) =
	\begin{cases}
		\sum\limits_{h\geq 0} z_i^{-1-h}z_j^h
		+\sum\limits_{n,m\geq0} \Big(\frac{(-1)^{n}\cdot \prod_{j=-m}^n \frac{1}{1+j\hbar}}{(m+n+1)\cdot m!\cdot n!}
		\cdot z_i^{-n-1} z_j^{-m-1}\Big),
		& \text{if}\ i<j,\\
		-\sum\limits_{h\geq 0} z_j^{-1-h}z_i^h
		+\sum\limits_{n,m\geq0} \Big(\frac{(-1)^{n}\cdot \prod_{j=-m}^n \frac{1}{1+j\hbar}}{(m+n+1)\cdot m!\cdot n!}
		\cdot z_i^{-n-1} z_j^{-m-1}\Big),
		& \text{if}\ i>j.
	\end{cases}
\end{align}
For convenience,
from now on,
for the $i$-th $\widehat{\vec{A}}(z_i,z_j)$ in the right-hand side of equation \eqref{eqn:vecHl as AA2},
when referring its explicit formula \eqref{eqn:vecH hat{A} formula2},
we use $h_i, m_i$, and $n_i$ to represent the corresponding indices in the sum.
Then for each subset $I\subseteq[l]=\{1,2,...,l\}$,
we can use the notation $\overrightarrow{\text{Cont}}_{I}$ to represent the contribution from the term involving $(m_i,n_i)$ for $i\in I$ and $(h_i)$ for $i\notin I$.
More precisely,
we consider
\begin{align}\label{eqn:def veccont}
	\begin{split}
	\overrightarrow{\text{Cont}}_{I}
	:=&(-1)^{l-1}
	\cdot [z_1^{-\mu_1-1}\cdots z_{l}^{-\mu_{l}-1}]\\
	&\ \ \sum_{\text{$l$-cycles } \sigma}
	\prod_{i\in I} \bigg(\sum\limits_{n_i,m_i\geq0} \Big(\frac{(-1)^{n_i}\cdot \prod_{j=-m}^n \frac{1}{1+j\hbar}}
	{(m_i+n_i+1)\cdot m_i!\cdot n_i!}
	\cdot z_{\sigma^{i}(1)}^{-n_i-1} z_{\sigma^{i+1}(1)}^{-m_i-1}\Big)\bigg)\\
	&\ \ \cdot\prod_{i\notin I}
	\Big(\delta_{\sigma^{i}(1)<\sigma^{i+1}(1)}
	\cdot\sum\limits_{h_i\geq 0} z_{\sigma^{i}(1)}^{-1-h_i}z_{\sigma^{i+1}(1)}^{h_i}
	-\delta_{\sigma^{i}(1)>\sigma^{i+1}(1)}
	\cdot\sum\limits_{h_i\geq 0} z_{\sigma^{i+1}(1)}^{-1-h_i}z_{\sigma^{i}(1)}^{h_i}\Big),
	\end{split}
\end{align}
where the Dirac symbol $\delta_{i>j}$ is $1$ if $i>j$, and $0$ if $i<j$,
and then
\begin{align*}
	\mu_1\cdots\mu_l\cdot\sum_{g=0}^\infty
	\big(\hbar^{2g-2+d+l} \cdot \vec{H}_{g;\mu}\big)
	=\sum_{I\subseteq [l]} \overrightarrow{\text{Cont}}_{I}.
\end{align*}
For each $I$,
by taking the sum of powers of all $z_i$ in the right-hand side of equation \eqref{eqn:def veccont},
we have that,
if indices $(m_i,n_i)_{i\in I}$ makes a non-zero contribution,
then
\begin{align*}
	\sum_{i\in I} (m_i+n_i+2) + \sum_{i\notin I} 1
	=\sum_{i=1}^l (\mu_i+1) = d+l,
\end{align*}
which is equivalent to the condition $d=\sum_{i\in I} (m_i+n_i+1)$.
Thus,
for the term having a simple pole at $\hbar=\frac{1}{d-1}$,
we must have $|I|=1$ and if we denote $I=\{j\}$ then we must have $(m_j,n_j)=(d-1,0)$.
Then by comparing the power of $z_{\sigma^{j+1}(1)}$ in the right-hand side of equation \eqref{eqn:def veccont},
for $\sigma$ and $(h_i)_{i\notin I}$ which make a non-zero contribution,
we must have
$\sigma^{j+1}(1)>\sigma^{j+2}(1)$ and $h_{j+1}$ satisfies
\begin{align*}
	-m_j-1+h_{j+1}=-\mu_{\sigma^{j+1}(1)}-1,
\end{align*}
where if $j+1>l$ then we denote $h_{j+1}:=h_{j+1-l}$.
By a similar analysis,
we must have
\begin{align}\label{eqn:sigma condition1}
	\sigma^{j+s}(1)>\sigma^{j+s+1}(1), \qquad s=1,...,l-1,
\end{align}
and
\begin{align*}
	h_{j+s}=d-1-\sum_{t=1}^{s}\mu_{\sigma^{j+t}(1)},
	\qquad s=1,...,l-1
\end{align*}
where we denote $h_{j+s}:=h_{j+s-l}$ if $j+s>l$.
Thus,
from equation \eqref{eqn:sigma condition1},
we must have $\sigma^{j+1}(1)=l, \sigma^{j+2}(1)=l-1,\dots, \sigma^{j+l}(1)=1$.
Then $\sigma$ is a $l$-cycle enforces $j=l$ and $\sigma=(l,l-1,...,1)$.
As a consequence,
\begin{align*}
	\overrightarrow{\text{Cont}}_{I}
	=\begin{cases}
		\frac{1}{d! \cdot \prod_{i=-d+1}^0(1+i\hbar)},&\text{\ if \ } I=\{l\} \text{\ for\ some\ }j,\\
		0,&\text{\ otherwise}.
	\end{cases}
\end{align*}
We thus have
\begin{align*}
	\vec{D}(\mu;d-1,1)
	=-(d-1)\cdot \Res_{\hbar=\frac{1}{d-1}}
	\bigg(\sum_{I\subseteq [l]} \overrightarrow{\text{Cont}}_{I}\bigg)
	=\frac{(d-1)^{d-2}}{d!\cdot (d-2)!}.
\end{align*}
Then from equation \eqref{eqn:vecC as vecDl},
the coefficient $\vec{C}(\mu;d-1,1)$ is given by
\begin{align*}
	\vec{C}((d);d-1,1)
	=2\cdot \vec{D}((d);d-1,1)
	=\frac{2\cdot (d-1)^{d-2}}{d!\cdot (d-2)!},
\end{align*}
which proves the theorem for $l\geq2$ case.
\end{proof}

\section{The large genus behavior of simple Hurwitz numbers}
\label{sec:Hurw}
In this section,
by employing the formula of connected $n$-point function for KP tau-function,
we prove Theorem \ref{thm:main thm for hurw},
which gives the structural property and large genus behavior of connected simple Hurwitz numbers.

\subsection{The KP integrability of simple Hurwitz numbers}
In this subsection,
we review the KP integrability of the generating function of ordinary simple Hurwitz numbers.
This is intensively studied in the literature (see, for examples, \cite{KL07, O00}).

It is beneficial to consider the following generating function of connected simple Hurwitz numbers,
\begin{align}\label{eqn:tau as expH}
	\tau(\hbar;\mathbf{t})
	=\exp\Big(
	\sum_{g\geq0 \atop \mu\in\mathcal{P}}
	\frac{\hbar^{2g-2+l(\mu)+d} \cdot H_{g;\mu}} {(2g-2+l(\mu)+d)!\cdot |{\rm Aut}(\mu)|}
	\cdot \prod_{i=1}^{l(\mu)} \big(\mu_i \cdot t_{\mu_i}\big)\Big),
\end{align}
where $|{\rm Aut}(\mu)|:=\prod_{j=1}^\infty m_j!$ if $m_j=\#\{i|\mu_i=j\}$.
This generating function $\tau(\hbar;\mathbf{t})$ can also be represented as a generating function of disconnected simple Hurwitz numbers (see, for example, \cite{GJV}).
In this paper,
we mainly study the structure of connected simple Hurwitz numbers.

From the Riemann--Hurwitz formula \eqref{eqn:RH formula},
the number $2g-2+l(\mu)+d$ appeared in equation \eqref{eqn:tau as expH} equals the number of branch points with simple ramification.
About the term $|{\rm Aut}(\mu)|$ in the denominator of above equation,
it comes from our definition for Hurwitz numbers, and see Remark \ref{rmk:explain autmu} for more details.

It is known that $\tau(\hbar;\mathbf{t})$ is a tau-function of the KP hierarchy (see \cite{O00}).
Moreover,
using representation theory of symmetric groups and equivalent description of Hurwitz numbers in terms of counting constellations,
the generating function $\tau(\hbar;\mathbf{t})$ admits the following formula in terms of Schur polynomials (see, for example, equation (4) in \cite{O00})
\begin{align}
	\tau(\hbar;\mathbf{t})
	=\sum_{\lambda\in\mathcal{P}}
	\exp(\hbar \cdot \kappa_{\lambda}) \cdot s_{\lambda}(\delta_{k,1}) \cdot s_{\lambda}(\mathbf{t}),
\end{align}
where $\kappa_{\lambda}$ is defined in equation \eqref{eqn:def kappa} and $s_{\lambda}(\delta_{k,1})$ can be computed by formula \eqref{eqn:schur hook}.
Thus,
from subsection \ref{sec:KP},
the affine coordinates corresponding to this tau-function $\tau(\hbar;\mathbf{t})$ are given by (see also \cite{KL07,O00})
\begin{align*}
	a_{n,m}
	=&(-1)^n\cdot \exp\Big(\frac{\hbar(m^2+m-n^2-n)}{2}\Big)
	\cdot s_{(m|n)}(\delta_{k,1})\\
	=&\frac{(-1)^{n}\cdot \exp\Big(\frac{\hbar(m^2+m-n^2-n)}{2}\Big)}{(m+n+1)\cdot m!\cdot n!}.
\end{align*}
When studying the expansion coefficients of $\log \tau(\hbar;\mathbf{t})$ using formulas \eqref{eqn:conn n=1} and \eqref{eqn:conn n},
we need the generating function of affine coordinates $A(z,w)$ defined in equation \eqref{eqn:A as a}.
Thus,
corresponding to the generating function $\tau(\hbar;\mathbf{t})$ of simple Hurwitz numbers, 
we have
\begin{align}\label{eqn:A for Hurw}
	A(z,w)
	=\sum_{n,m\geq0} \bigg(\frac{(-1)^{n}\cdot \exp\Big(\frac{\hbar(m^2+m-n^2-n)}{2}\Big)}{(m+n+1)\cdot m!\cdot n!}
	 \cdot z^{-n-1} w^{-m-1}\bigg).
\end{align}

\subsection{The structural property and large genus behavior of simple Hurwitz numbers}
In this subsection,
we use the KP integrability of $\tau(\hbar;\mathbf{t})$ to study the connected simple Hurwitz numbers,
which proves Theorem \ref{thm:main thm for hurw}.

\begin{thm}[=Theorem \ref{thm:main thm for hurw}]
	For the connected simple Hurwitz numbers,
	we fix the degree $d\geq2$ and the ramification type $\mu=(\mu_1,...,\mu_{l})$,
	which is a partition of $d$.
	Then there are some integers $C(\mu,k)\in\mathbb{Z}, 1\leq k\leq \binom{d}{2}$ such that,
	for arbitrary genus $g\geq0$,
	\begin{align}\label{eqn:H as C}
		H_{g;\mu}
		=\frac{2}{d! \cdot \mu_1 \cdots \mu_{l}}
		\cdot \sum_{k=1}^{\binom{d}{2}}
		\big(C(\mu;k) \cdot k^{d+2g-2+l}\big).
	\end{align}
	Moreover,
	these coefficients satisfy
	$C(\mu;k)=0$ for $\binom{d-1}{2}<k<\binom{d}{2}$
	and
	\begin{align}
		C(\mu;\binom{d}{2})=1,
		\qquad \qquad C(\mu;\binom{d-1}{2})=-d\cdot \#\{i\ |\ \mu_i=1\}.
	\end{align}
	Thus,
	for fixed $\mu$,
	as $g\rightarrow\infty$,
	\begin{align*}
		H_{g;\mu}
		=\frac{2}{d!\cdot \mu_1\cdots \mu_l}
		\cdot\binom{d}{2}^{d+2g-2+l}
		-&\frac{2\cdot \#\{i\ |\ \mu_i=1\}}{(d-1)!\cdot \mu_1\cdots \mu_l}\cdot\binom{d-1}{2}^{d+2g-2+l}\\
		&\qquad+O\Bigg(\bigg(\binom{d-1}{2}-1\bigg)^{d+2g-2+l}\Bigg).
	\end{align*}
\end{thm}
\begin{proof}
First,
we deal with the $\mu=(d)$ case.
Applying the formula \eqref{eqn:conn n=1} of connected one-point function to the tau-function $\tau(\hbar;\mathbf{t})$ defined in equation \eqref{eqn:tau as expH},
we obtain
\begin{align*}
	\sum_{g=0}^\infty 
	\frac{\hbar^{2g-1+d}}{(2g-1+d)!} H_{g;(d)}
	=&\frac{1}{d}
	\cdot [z^{-d-1}]
	\ \sum_{i=1}^\infty
	\bigg(z^{-i-1}\cdot \frac{\partial \log\tau(\hbar;\mathbf{t})}{\partial t_i} \big|_{\mathbf{t}=0}\bigg)\\
	=&\frac{1}{d}
	\cdot [z^{-d-1}]\ A(z,z)\\
	=&\frac{1}{d}
	\cdot \sum_{n,m\geq0 \atop n+m=d-1} 
	\frac{(-1)^n 
		\cdot\exp\big(\frac{\hbar(m^2+m-n^2-n)}{2}\big)}
		{(m+n+1)\cdot m!\cdot n!},
\end{align*}
where we have used the formula \eqref{eqn:A for Hurw} for $A(z,w)$.
One can notice that the last line of the above equation is a linear combination of functions $\exp(\hbar k)$ with $-\binom{d}{2}\leq k \leq \binom{d}{2}$.
We thus have the following expression
\begin{align}\label{eqn:H 1point as D}
	\sum_{g=0}^\infty 
	\frac{\hbar^{2g-1+d}}{(2g-1+d)!} H_{g;(d)}
	=&\sum_{k=-\binom{d}{2}}^{\binom{d}{2}}
	\big(D((d);k) \cdot \exp(\hbar k)\big),
\end{align}
where the coefficients $D((d);k)$ are given by
\begin{align}\label{eqn:D1}
	D((d);k)
	=\frac{1}{d}
	\cdot \sum_{n,m\geq0, n+m=d-1\atop m^2+m-n^2-n=2k} 
	\frac{(-1)^n}{(m+n+1)\cdot m!\cdot n!}.
\end{align}
The left-hand side of equation \eqref{eqn:H 1point as D},
as a function of $\hbar$,
is an odd function if $d$ is even and an even function if $d$ is odd.
Thus, the coefficients $D((d);k)$ must satisfy the following condition
\begin{align}\label{eqn:D=+-D 1}
	D((d);k)=(-1)^{d-1} \cdot D((d);-k)
\end{align}
for all possible $k$.

By taking Taylor expansion and taking the coefficient of $\hbar^{2g-1+d}$ in both sides of equation \eqref{eqn:H 1point as D} with the condition $d\geq2$,
we obtain a formula for the connected simple Hurwitz number as
\begin{align*}
	H_{g;(d)}
	=\frac{2}{d! \cdot d}
	\cdot \sum_{k=1}^{\binom{d}{2}}
	\big(C((d);k) \cdot k^{2g-1+d}\big),
\end{align*}
where the coefficients $C((d);k)$ are obtained from $D((d);k)$ in terms of
\begin{align*}
	C((d),k)
	=\frac{d!\cdot d}{2} \cdot \Big(D((d),k)+(-1)^{2g-1+d}D((d),-k)\Big)
	=d!\cdot d \cdot D((d),k).
\end{align*}
In the second equal sign of above equation, we have used the equation \eqref{eqn:D=+-D 1}.
From equation \eqref{eqn:D1} computing $D((d);k)$,
we know $C((d),k)\in\mathbb{Z}$ since $\frac{d! \cdot d}{d^2 \cdot m!\cdot n!}$ is always an integer with the conditions $m, n\geq0, m+n=d-1$.
 
Moreover,
from equation \eqref{eqn:D1},
if $k=\binom{d}{2}$,
then the possible $(m,n)$ in the summation which contributes to $D((d),k)$ is $(m,n)=(d-1,0)$.
Then we have
\begin{align*}
	C((d);\binom{d}{2})=d!\cdot d \cdot D((d),k)=1.
\end{align*}
If $\binom{d-1}{2}\leq k<\binom{d}{2}$,
then there is no $(m, n)$ in the summation, which can make a non-zero contribution to the right-hand side of equation \eqref{eqn:D1}.
Thus
$$C((d);k)=d!\cdot d \cdot D((d),k)=0,
\quad \binom{d-1}{2}\leq k<\binom{d}{2}.$$
The above discussion proves the theorem in this case.

Next,
we deal with the general $l=l(\mu)\geq2$ case.
From definition \eqref{eqn:tau as expH} of the tau-function $\tau(\hbar;\mathbf{t})$,
we have,
\begin{align*}
	\sum_{g=0}^\infty 
	\frac{\hbar^{2g-2+d+l}}{(2g-2+d+l)!}
	&H_{g;\mu}\\
	=\frac{1}{\mu_1\cdots \mu_l}
	&\cdot [z_1^{-\mu_1-1}\cdots z_{l}^{-\mu_{l}-1}]
	\ \sum_{j_1,\cdots,j_l\geq 1}
	\bigg(\frac{\partial^l \log \tau(\hbar;\mathbf{t})}{\partial t_{j_1} \cdots \partial t_{j_l}} \Big|_{\mathbf{t}=0}
	\cdot z_1^{-j_1-1} \cdots z_l^{-j_l-1}\bigg).
\end{align*}
One can then use the formula \eqref{eqn:conn n} of connected $n$-point function to compute the right-hand side of above equation to obtain
\begin{align}\label{eqn:Hl as AA}
	\begin{split}
	\sum_{g=0}^\infty 
	\frac{\hbar^{2g-2+d+l}}{(2g-2+d+l)!}
	&H_{g;\mu}\\
	=\frac{(-1)^{l-1}}{\mu_1\cdots\mu_l}
	&\cdot [z_1^{-\mu_1-1}\cdots z_{l}^{-\mu_{l}-1}]
	\ \sum_{\text{$l$-cycles } \sigma}
	\Big( \prod_{i=1}^l \widehat A (z_{\sigma^i(1)},z_{\sigma^{i+1}(1)}) \Big),
	\end{split}
\end{align}
where $\widehat A(z_i,z_j)$ in this case is explicitly given by
\begin{align}\label{eqn:H hat{A} formula}
	\widehat A(z_i,z_j) =
	\begin{cases}
		\sum\limits_{h\geq 0} z_i^{-1-h}z_j^h
		+\sum\limits_{n,m\geq0} \Big(\frac{(-1)^{n}\cdot \exp\big(\frac{\hbar(m^2+m-n^2-n)}{2}\big)}{(m+n+1)\cdot m!\cdot n!}
		\cdot z_i^{-n-1} z_j^{-m-1}\Big),
		& \text{if}\ i<j,\\
		-\sum\limits_{h\geq 0} z_j^{-1-h}z_i^h
		+\sum\limits_{n,m\geq0} \Big(\frac{(-1)^{n}\cdot \exp\big(\frac{\hbar(m^2+m-n^2-n)}{2}\big)}{(m+n+1)\cdot m!\cdot n!}
		\cdot z_i^{-n-1} z_j^{-m-1}\Big),
		& \text{if}\ i>j.
	\end{cases}
\end{align}
For convenience,
for the $i$-th $\widehat A (z_{\sigma^{i}(1)},z_{\sigma^{i+1}(1)})$ in the right-hand side of equation \eqref{eqn:Hl as AA},
when referring its explicit formula \eqref{eqn:H hat{A} formula},
we will use $h_i, m_i, n_i$ to represent their indices in the sum.

Similar to the one-point case,
after expanding the product of all $\widehat A (z_{\sigma^i(1)},z_{\sigma^{i+1}(1)})$,
taking the summation over $\sigma\in\{l-\text{cycles}\}$,
and taking the coefficient of $z_1^{-\mu_1-1}\cdots z_{l}^{-\mu_{l}-1}$,
we know that the right-hand side of equation \eqref{eqn:Hl as AA} will be a linear combination of the functions $\exp(\hbar k), k\in\mathbb{Z}$.
Thus,
we can simplify assume the following expansion formula for generating function of $H_{g;\mu}$ as
\begin{align}\label{eqn:H lpoint as D}
	\sum_{g=0}^\infty 
	\frac{\hbar^{2g-2+d+l}}{(2g-2+d+l)!} H_{g;\mu}
	=&\sum_{k\in\mathbb{Z}}
	\big(D(\mu;k) \cdot \exp(\hbar k)\big),
\end{align}
where the coefficients $D(\mu;k)$ are some rational numbers.
Actually,
as we will see,
for fixed partition $\mu$, there are only finitely many $(h_i, m_i, n_i)$ which make a non-zero contribution to the right-hand side of equation \eqref{eqn:Hl as AA}.
Thus $D(\mu;k)\neq0$ only for finitely many $k$.
Below,
we give a detailed analysis on these coefficients $D(\mu;k)$.

By applying the transform $\hbar\rightarrow-\hbar$ to the left-hand side of equation \eqref{eqn:H lpoint as D},
the first property satisfied by $D(\mu;k)$ reads
\begin{align}\label{eqn:Dmuk=-Dmu-k}
	D(\mu;k)=(-1)^{d+l} \cdot D(\mu;-k).
\end{align}
When expanding the right-hand side of equation \eqref{eqn:Hl as AA},
the formula \eqref{eqn:H hat{A} formula} tells $\widehat A(z_{\sigma^{i}(1)},z_{\sigma^{i+1}(1)})$ will contribute two kinds of terms, $\sum_{h_i\geq0}(\cdot)$ or $\sum_{n_i,m_i\geq0}(\cdot)$.
To give a clear description to $D(\mu;k)$,
we need to record which kind of term of $\widehat A(z_{\sigma^{i}(1)},z_{\sigma^{i+1}(1)})$ contributes for each $1\leq i\leq l$.
Thus,
for each $\sigma\in\{l-\text{cycles}\}$ and subset $I$ of $[l]:=\{1,2,...,l\}$,
we use the notation $\text{Cont}_{\sigma,I}$ to denote the contribution form the term,
which involves $(m_i, n_i)$ when $i\in I$,
and involves $h_i$ when $i\notin I$.
More precisely,
\begin{align}\label{eqn:def cont}
	\begin{split}
	\text{Cont}_{\sigma,I}
	:=&\frac{(-1)^{l-1}}{\mu_1\cdots\mu_l}
	\cdot [z_1^{-\mu_1-1}\cdots z_{l}^{-\mu_{l}-1}]\\
	&\ \ \prod_{i\in I} \bigg(\sum\limits_{n_i,m_i\geq0} \Big(\frac{(-1)^{n_i}\cdot \exp\big(\frac{\hbar(m_i^2+m_i-n_i^2-n_i)}{2}\big)}{(m_i+n_i+1)\cdot m_i!\cdot n_i!}
	\cdot z_{\sigma^{i}(1)}^{-n_i-1} z_{\sigma^{i+1}(1)}^{-m_i-1}\Big)\bigg)\\
	&\ \ \cdot\prod_{i\notin I}
	\Big(\delta_{\sigma^{i}(1)<\sigma^{i+1}(1)}
	\cdot\sum\limits_{h_i\geq 0} z_{\sigma^{i}(1)}^{-1-h_i}z_{\sigma^{i+1}(1)}^{h_i}
	-\delta_{\sigma^{i}(1)>\sigma^{i+1}(1)}
	\cdot\sum\limits_{h_i\geq 0} z_{\sigma^{i+1}(1)}^{-1-h_i}z_{\sigma^{i}(1)}^{h_i}\Big),
	\end{split}
\end{align}
where the Dirac symbol $\delta_{i>j}$ is $1$ if $i>j$, and $0$ if $i<j$.
Thus, equation \eqref{eqn:Hl as AA} can be expressed as
\begin{align}\label{eqn:H as cont}
	\sum_{g=0}^\infty 
	\frac{\hbar^{2g-2+d+l}}{(2g-2+d+l)!} H_{g;\mu}
	=&\sum_{\sigma\in\{l-\text{cycles}\} \atop I\subseteq[l]} \text{Cont}_{\sigma,I}.
\end{align}
Moreover,
by a similar idea,
we can decompose those coefficients $D(\mu;k)$ in equation \eqref{eqn:H lpoint as D} to different contributions as
\begin{align*}
	D(\mu;k)
	=\sum_{\sigma\in\{l-\text{cycles}\} \atop I\subseteq[l]} D(\mu;k)_{\sigma,I},
\end{align*}
where $D(\mu;k)_{\sigma,I}$ is the coefficient of $\exp(\hbar k)$ in $\text{Cont}_{\sigma,I}$.

It is time to give a precise description for $D(\mu;k)$.
We first deal with the case of $I=[l]$.
For each given $\sigma\in\{l-\text{cycles}\}$,
if the numbers $(m_i,n_i)_{i\in [l]}$ make a non-zero contribution to $\text{Cont}_{\sigma,[l]}$,
then from equation \eqref{eqn:def cont},
they must satisfy the condition
\begin{align*}
	m_i+n_{i+1}=\mu_{\sigma^{i+1}(1)}-1,
	\quad 1\leq i \leq l,
\end{align*}
where $n_{l+1}:=n_1$.
Recall that the exponent $k$ is determined from $(m_i,n_i)_{i\in [l]}$ by
$$k=\frac{1}{2}\sum_{i=1}^l (m_i^2+m_i-n_i^2-n_i).$$
Thus,
in this case,
the maximal $k$ such that $D(\mu;k)_{\sigma,[l]}\neq0$ satisfies
\begin{align*}
	k\leq\frac{1}{2}\sum_{i=1}^l (m_i^2+m_i)
	\leq\frac{1}{2}\sum_{i=1}^l (\mu_i^2-\mu_i)
	\leq \binom{d-1}{2}.
\end{align*}
Moreover,
if $l=2$ and $\mu_2=1$,
we have $\sigma=(1,2)$ and the maximal $k$ such that $D(\mu;k)_{\sigma,[l]}\neq0$ is $\binom{d-1}{2}$,
which comes from $(m_1=\mu_1-1, m_2=n_1=n_2=0)$.
In addition to this case,
the maximal $k$ such that $D(\mu;k)_{\sigma,[l]}\neq0$ is less than $\binom{d-1}{2}$.
Thus, when $\mu=(d-1,1)$,
we have
\begin{align}\label{eqn:mu=(d-1,1) I=[l]}
	D((d-1,1);k)_{(1,2),[2]}
	=\begin{cases}
		\frac{-1}{\mu_1 \cdot \mu_2 \cdot (m_1+n_1+1) \cdot m_1!}
		=-\frac{1}{(d-1)\cdot (d-1)!},
		&\text{if}\ k=\binom{d-1}{2},\\
		0, &\text{if}\ k>\binom{d-1}{2}. 
	\end{cases}
\end{align}
Otherwise, i.e., $\mu\neq(d-1,1)$,
we have that $D(\mu;k)_{\sigma,[l]}=0$ for all $k\geq\binom{d-1}{2}$.

Similarly,
if $I$ is a proper subset of $[l]$,
from equation \eqref{eqn:Hl as AA},
the condition for $D(\mu;k)_{\sigma,I}\neq0$ will enforce
\begin{align}\label{eqn:k as mn I}
	k=\frac{1}{2}\sum_{i\in I} (m_i^2+m_i-n_i^2-n_i)
\end{align}
and for $1\leq i \leq l$,
\begin{align}\label{eqn:9cases}
	\begin{cases}
		m_i+n_{i+1}=\mu_{\sigma^{i+1}(1)}-1, &\text{if}\ i, i+1\in I,\\
		m_i+h_{i+1}=\mu_{\sigma^{i+1}(1)}-1, &\text{if}\ i\in I, i+1\notin I, \sigma^{i+1}(1)<\sigma^{i+2}(1),\\
		m_i-h_{i+1}=\mu_{\sigma^{i+1}(1)}, &\text{if}\ i\in I, i+1\notin I, \sigma^{i+1}(1)>\sigma^{i+2}(1),\\
		n_{i+1}-h_{i}=\mu_{\sigma^{i+1}(1)}, &\text{if}\ i\notin I, i+1\in I, \sigma^{i}(1)<\sigma^{i+1}(1),\\
		n_{i+1}+h_{i}=\mu_{\sigma^{i+1}(1)}-1, &\text{if}\ i\notin I, i+1\in I, \sigma^{i}(1)>\sigma^{i+1}(1),\\
		h_{i+1}-h_{i}=\mu_{\sigma^{i+1}(1)}, &\text{if}\ i\notin I, i+1\notin I, \sigma^{i}(1)<\sigma^{i+1}(1), \sigma^{i+1}(1)<\sigma^{i+2}(1),\\
		-h_{i+1}-h_{i}=\mu_{\sigma^{i+1}(1)}+1, &\text{if}\ i\notin I, i+1\notin I, \sigma^{i}(1)<\sigma^{i+1}(1), \sigma^{i+1}(1)>\sigma^{i+2}(1),\\
		h_{i+1}+h_{i}=\mu_{\sigma^{i+1}(1)}-1, &\text{if}\ i\notin I, i+1\notin I, \sigma^{i}(1)>\sigma^{i+1}(1), \sigma^{i+1}(1)<\sigma^{i+2}(1),\\
		-h_{i+1}+h_{i}=\mu_{\sigma^{i+1}(1)}, &\text{if}\ i\notin I, i+1\notin I, \sigma^{i}(1)>\sigma^{i+1}(1), \sigma^{i+1}(1)>\sigma^{i+2}(1),
	\end{cases}
\end{align}
where if $i=l$, we regard $i+1$ as $1$.
Notice that the $7$-th case in above equation, $``i\notin I, i+1\notin I, \sigma^{i}(1)<\sigma^{i+1}(1), \sigma^{i+1}(1)>\sigma^{i+2}(1)"$, does not contribute since $h_i, h_{i+1}\geq0$ and $-h_{i+1}-h_{i}=\mu_{\sigma^{i+1}(1)}+1$ could not be satisfied simultaneously.
Thus, $I$ cannot be an empty set.
In addition,
for the case whose contribution comes from the subset $I$,
by considering the sum of all exponents of $z_i, i\in [l]$ in both sides of equation \eqref{eqn:Hl as AA},
we have
\begin{align}\label{eqn:|mu|=mn}
	d=\sum_{i\in I} (m_i+n_i+1).
\end{align}
Below,
by using equations \eqref{eqn:k as mn I}, \eqref{eqn:9cases} and \eqref{eqn:|mu|=mn},
we give a precise description for four cases of $k$.

Case i). For $k>\binom{d}{2}$,
from equation \eqref{eqn:|mu|=mn},
we have $\sum_{i\in I}m_i\leq d-1$.
Thus there is no suitable $(m_i,n_i)_{i\in I}$ such that $k=\frac{1}{2}\sum_{i\in I} (m_i^2+m_i-n_i^2-n_i)$ is greater than $\binom{d}{2}$.

Case ii). For $k=\binom{d}{2}$,
since from equation \eqref{eqn:|mu|=mn} we have $\sum_{i\in I}m_i\leq d-|I|$,
the subset $I$ must have only one single element and we denote it by $I=\{j\}$.
Moreover,
from
$k=\frac{1}{2} (m_j^2+m_j-n_j^2-n_j)=\binom{d}{2}$,
we must have $m_j=d-1$ and $n_j=0$.
For the $i=j$-th term in the right-hand side of equation \eqref{eqn:Hl as AA},
since $j\in I, j+1 \notin I$ and $h_{j+1}\geq0$,
the $2$-nd case in equation \eqref{eqn:9cases} cannot be satisfied.
Thus the non-zero contribution must come from the $3$-rd case in equation \eqref{eqn:9cases},
i.e.,
we have $\sigma^{j+1}(1)>\sigma^{j+2}(1)$ and $h_{j+1}=d-\mu_{\sigma^{j+1}(1)}-1$.
When $l=l(\mu)=2$, there is only one $2$-cycle $\sigma=(1,2)$,
and from the condition $\sigma^{j+1}(1)>\sigma^{j+2}(1)$,
we have $\sigma^{j+1}(1)=2$ and $\sigma^{j+2}(1)=1$.
Thus, $\sigma$ is a $2$-cycle gives $j=2$, and $(m_2,n_2)=(d-1,0), h_1=d-\mu_2-1=\mu_1-1$.
The result in this case is
\begin{align}\label{eqn:D l=2}
	D(\mu;k)_{\sigma,I}
	=D((\mu_1,\mu_2);\binom{d}{2})_{(1,2),\{2\}}
	=\frac{-1}{\mu_1 \cdot \mu_2}
	\cdot \frac{1}{d\cdot (d-1)!}\cdot(-1)
	=\frac{1}{d!\cdot \mu_1 \cdot \mu_2}.
\end{align}
When $l=l(\mu)>2$,
for the $i=(j+1)$-th term in the right-hand side of equation \eqref{eqn:Hl as AA},
since we already know $j+1, j+2\notin I, \sigma^{j+1}(1)>\sigma^{j+2}(1)$ and $h_{j+2}\geq0$,
the non-zero contribution must come from the $9$-th case in equation \eqref{eqn:9cases}.
Thus, we must have $\sigma^{j+2}(1)>\sigma^{j+3}(1)$
and $h_{j+2}=d-\mu_{\sigma^{j+1}(1)}-\mu_{\sigma^{j+2}(1)}-1$.
By a similar analysis,
we have,
the $l$-cycle $\sigma$ must satisfy
\begin{align}\label{eqn:sigma 2}
	\sigma^{j+s}(1)>\sigma^{j+s+1}(1),
	\quad s=1,2,...,l-1,
\end{align}
and the $(h_{i})_{i\in[l]\setminus{j}}$ are determined by
\begin{align*}
	h_{j+s}=d-1-\sum_{t=1}^{s}\mu_{\sigma^{j+t}(1)}, \quad s=1,2,...,l-1,
\end{align*}
where $h_{j+s}:=h_{j+s-l}$ if $j+s>l$.
As a consequence,
from equation \eqref{eqn:sigma 2},
we must have $\sigma^{j+1}(1)=l, \sigma^{j+2}(1)=l-1,..., \sigma^{j+l}(1)=1$,
which corresponds to the fact that the only possibility is $j=l$ and $\sigma=(l,l-1,...,1)$.
The result in this case is then
\begin{align}\label{eqn:D case2}
	D(\mu;\binom{d}{2})_{\sigma,I}
	=\begin{cases}
		\frac{(-1)^{l-1}}{\mu_1\cdots\mu_l}
		\cdot \frac{1}{d\cdot (d-1)!}\cdot(-1)^{l-1}
		=\frac{1}{d! \cdot \mu_1\cdots\mu_l}  & \text{if}\ I=\{l\}, \sigma=(l,l-1,...,1),\\
		0 & \text{otherwise}.
	\end{cases}
\end{align}
This formula also holds for the $l=2$ case discussed in equation \eqref{eqn:D l=2}.

Case iii). For $\binom{d-1}{2}<k<\binom{d}{2}$,
similar to the Case i),
if there is $j\in I$ such that $n_j\neq0$,
then $\sum_{i\in I}m_i\leq d-2$ and $k=\frac{1}{2}\sum_{i\in I} (m_i^2+m_i-n_i^2-n_i)$ must be less than or equal to $\binom{d-1}{2}$.
And even we assume $n_i=0$ for all $i\in I$,
then $|I|=1$ reduces to the Case ii),
and $|I|\geq2$ gives $\sum_{i\in I}m_i\leq d-2$,
which also results in $k\leq \binom{d-1}{2}$.
As a consequence, there is no suitable $I$ and $(m_i,n_i)_{i\in I}$ which makes a non-zero contribution to $D(\mu;k)$ with $\binom{d-1}{2}<k<\binom{d}{2}$.

Case iv). For $k=\binom{d-1}{2}$,
if there is some $j$ such that $n_j\neq0$,
from equation \eqref{eqn:|mu|=mn},
we have $\sum_{i\in I}m_i\leq d-1-|I|$.
Then $k=\frac{1}{2}\sum_{i\in I} (m_i^2+m_i-n_i^2-n_i)<\binom{d-1}{2}$.
Thus, to obtain a non-zero $D(\mu;\binom{d-1}{2})_{\sigma;I}$,
we still need $n_i=0$ for all $i\in I$.
The case of $|I|=1$ reduces to the Case ii),
and thus we can assume $|I|\geq2$.
Then $\sum_{i\in I}m_i=d-|I|$ and $k=\frac{1}{2}\sum_{i\in I} (m_i^2+m_i)=\binom{d-1}{2}$ enforce $|I|=2$.
We denote $|I|=\{j,j'\}$ with $j<j'$,
then we must have $(m_j,m_{j'})=(d-2,0)$ or $(m_j,m_{j'})=(0,d-2)$.
The case of $|I|=l=2$ has been discussed in equation \eqref{eqn:mu=(d-1,1) I=[l]},
so we can now assume $l\geq3$.
We first deal with the $(m_j,m_{j'})=(d-2,0)$ case.
If $j'=j+1$,
then when considering the $i=j$-th term in the right-hand side of equation \eqref{eqn:Hl as AA},
we need to use the $1$-st case in equation \eqref{eqn:9cases}.
However, $m_j+n_{j+1}=d-2$ makes a contradiction with $m_j+n_{j+1}=\mu_{\sigma^{j+1}(1)}-1$.
Thus we know $j'>j+1$ and $j+1\notin I$.
Moreover, the $2$-nd case of equation \eqref{eqn:9cases} is impossible since $h_{j+1}\geq0$.
Then $h_{j+1}=d-2-\mu_{\sigma^{j+1}(1)}$ is derived from the $3$-rd case of equation \eqref{eqn:9cases}.
Similar to the discussions in Case ii),
we finally have,
the $l$-cycle $\sigma$ must satisfy
\begin{align*}
	\sigma^{j+s}(1)>\sigma^{j+s+1}(1),\quad s=1,2,...,j'-j-1,
\end{align*}
and the $(h_{i})_{i\in\{j+1,...,j'\}}$ are determined by
\begin{align*}
	h_{j+s}=d-2-\sum_{t=1}^{s}\mu_{\sigma^{j+t}(1)}, \quad s=1,2,...,j'-j-1.
\end{align*}
When considering the $i=(j'-1)$-th term in the right-hand side of equation \eqref{eqn:Hl as AA},
we need to use the $5$-th case of equation \eqref{eqn:9cases} since $j'-1\notin I, j'\in I$ and $h_{j'-1}\geq0$.
Then we obtain $\sigma^{j'-1}(1)>\sigma^{j'}(1)$ and 
\begin{align*}
	d-2-\sum_{t=1}^{j'-j-1}\mu_{\sigma^{j+t}(1)}
	=\mu_{\sigma^{j'}(1)}-1.
\end{align*}
Thus,
we must have $j'-j-1=l(\mu)-2=l-2$ and $\mu_{\sigma^{j'+1}(1)}=1$,
which gives $j'=l, j=1$ and $\sigma$ should satisfy $\sigma^2(1)>\sigma^3(1)>\cdots>\sigma^l(1)$.
Moreover,
from the $i=l$-th term in the right-hand side of equation \eqref{eqn:Hl as AA} and the $1$-st case of equation \eqref{eqn:9cases},
we have $0=\mu_{\sigma^1(1)}-1$.
Finally,
if $\mu=(\mu_1,...,\mu_l)$ satisfies $\mu_l>1$,
then there is no such $\sigma$ which can make a non-zero contribution.
If we assume $s\leq l$ is the minimal number such that $\mu_s=1$,
then $\sigma(1)$ must be greater than or equal to $s$,
and additionally $\sigma(1)$ could not be 1 since $\sigma$ is a $l$-cycle. 
Thus, 
the possible $\sigma$ are given by
\begin{align}\label{eqn:l-s sigma}
	\sigma=(a,l,l-1,...,\hat{a},...,1)
	\  \text{with}\  \text{max}\{s,2\}\leq a\leq l.
\end{align}
There are a total of $(l-s+\delta_{s\neq1})$ possible $\sigma$ here.

One can deal with the case of $(m_j,m_{j'})=(0,d-2)$ similarly and obtain the following conditions:
\begin{align*}
	j=j'-1,
	\qquad \sigma^{j'+1}(1)>\sigma^{j'+2}(1)>\cdots>\sigma^{j+l}(1),
	\qquad \mu_{\sigma^{j'}(1)}=1.
\end{align*}
Since $\sigma^{j+l}(1)=\sigma^j(1)\neq1$,
we must have $j'=l, j=l-1$,
$\sigma=(l,l-1,...,1)$,
and $\mu_1=1$.
That is to say,
only when $\mu=(1^d)=(1,1,...,1)$,
there is exactly one such $\sigma=(l,l-1,...,1)$ which makes a non-zero contribution.
Combing with the $(m_j,m_{j'})=(d-2,0)$ case discussed near equation \eqref{eqn:l-s sigma},
we have
\begin{align}\label{eqn:d-12}
	D(\mu;\binom{d-1}{2})
	=\sum_{\sigma\in\{l-\text{cycles}\} \atop I\subseteq[l]} D(\mu;\binom{d-1}{2})_{\sigma,I}
	=-\frac{l-s+1}{(d-1)!\cdot \mu_1\cdots\mu_l},
\end{align}
where $s$ is the minimal number such that $\mu_s=1$, and if $\mu_l\neq1$, then $s:=l+1$.
We notice that the equation \eqref{eqn:d-12} also holds for the $l=l(\mu)=2$ case which is discussed in equation \eqref{eqn:mu=(d-1,1) I=[l]}.

In conclusion,
we take the coefficient of $\hbar^{2g-2+d+l}$ in both sides of equation \eqref{eqn:H lpoint as D} to obtain
\begin{align*}
	H_{g;\mu}
	=\frac{2}{d! \cdot \mu_1 \cdots \mu_{l}}
	\cdot \sum_{k=1}^{\infty}
	\big(C(\mu;k) \cdot k^{d+2g-2+l}\big),
\end{align*}
where the coefficients $C(\mu;k)$ are given by
\begin{align}\label{eqn:C as D}
	C(\mu;k)
	=\frac{d!\cdot \mu_1 \cdots \mu_{l}}{2} \cdot
	\big(D(\mu;k)+(-1)^{2g-2+d+l}D(\mu;-k)\big)
	=d!\cdot \mu_1 \cdots \mu_{l} \cdot D(\mu;k).
\end{align}
In the last equal sign of above equation,
we have used equation \eqref{eqn:Dmuk=-Dmu-k}.
First,
from equation \eqref{eqn:|mu|=mn},
we know that $C(\mu;k)$ must be an integer by comparing equations \eqref{eqn:def cont}, \eqref{eqn:C as D} and the fact
\begin{align*}
	d! \cdot\frac{1}{\prod_{i\in I} (m_i+n_i+1) m_i! n_i!}
	=\frac{\sum_{i\in I} (m_i+n_i+1)}{\prod_{i\in I} (m_i+n_i+1) m_i! n_i!}
	\in\mathbb{Z}.
\end{align*}
Then for $k>\binom{d}{2}$ or $\binom{d-1}{2}<k<\binom{d}{2}$,
from the Cases i) and iii) discussed above,
there is no suitable $\sigma$ and $I$ which makes a non-zero contribution,
thus, in these cases,
\begin{align*}
	C(\mu;k)
	=d!\cdot \mu_1 \cdots \mu_{l} \cdot D(\mu;k)
	=0.
\end{align*}
For $k=\binom{d}{2}$,
from the Case ii) discussed above,
the contribution comes from only $I=\{l\}$ and $\sigma=(l,l-1,...,1)$.
Then from equation \eqref{eqn:D case2},
we have
\begin{align*}
	C(\mu;\binom{d}{2})
	=d!\cdot \mu_1 \cdots \mu_{l} \cdot D(\mu;\binom{d}{2})_{(l,l-1,...,1),\{l\}}
	=1.
\end{align*}
For $k=\binom{d-1}{2}$,
from the Case iv) discussed above,
we assume $d\geq3$,
then the contribution comes from a total of $\#\{i\ |\ \mu_i=1\}$ possible $\sigma$ and corresponding $I$.
From equation \eqref{eqn:d-12},
we have
\begin{align*}
	C(\mu;\binom{d-1}{2})
	=d!\cdot \mu_1 \cdots \mu_{l}
	\cdot \frac{-\#\{i\ |\ \mu_i=1\}} {(d-1)!\cdot\mu_1\cdots\mu_l}
	=-d\cdot \#\{i\ |\ \mu_i=1\}.
\end{align*}
This theorem is thus proved.
\end{proof}

\begin{rmk}
	The existence of equation \eqref{eqn:H as C} and the formula for $C(\mu;\binom{d}{2})$ were obtained by Do, He and Robertson as Theorem 1.5 and Theorem 1.6 in \cite{DHR}, respectively.
	The vanishing of $C(\mu;k)$ for $\binom{d-1}{2}<k<\binom{d}{2}$ was also conjectured by them as Conjecture 4.1 in \cite{DHR}.
\end{rmk}

\begin{ex}
	We provide two examples of the $l=l(\mu)=1$ case.
	When $\mu=(5)$,
	we have
	\begin{align*}
		\sum_{g=0}^\infty 
		\frac{\hbar^{2g+4}}{(2g+4)!} H_{g;(5)}
		=&\frac{1}{5!\cdot 5}
		\big({\mathrm e}^{10 \hbar}-4 {\mathrm e}^{5 \hbar}+6-4 {\mathrm e}^{-5 \hbar}+{\mathrm e}^{-10 \hbar}\big).
	\end{align*}
	Then, after taking the coefficient of $\hbar^{2g+4}$ in both sides of above equation,
	we obtain
	\begin{align*}
		H_{g;(5)}
		=\frac{2}{5!\cdot 5}
		\big(10^{2g+4}-4\cdot5^{2g+4}\big),
	\end{align*}
	which matches with data in Table 2 of \cite{DHR}.
	
	When $\mu=(10)$,
	we have
	\begin{align*}
		\sum_{g=0}^\infty 
		\frac{\hbar^{2g+9}}{(2g+9)!} H_{g;(10)}
		=&\frac{1}{10!\cdot 10}
		\cdot 
		\big({\mathrm e}^{45 \hbar}-9 {\mathrm e}^{35 \hbar}+36 {\mathrm e}^{25 \hbar}-84 {\mathrm e}^{15 \hbar}+126 {\mathrm e}^{5 \hbar}-126 {\mathrm e}^{-5 \hbar}\\
		&+84 {\mathrm e}^{-15 \hbar}-36 {\mathrm e}^{-25 \hbar}+9 {\mathrm e}^{-35 \hbar}-{\mathrm e}^{-45 \hbar}\big),
	\end{align*}
	and
	\begin{align*}
		H_{g;(10)}
		=&\frac{2}{10!\cdot 10}
		\cdot 
		\big(45^{2g+9}-9\cdot35^{2g+9}+36\cdot25^{2g+9}
		-84\cdot15^{2g+9}+126\cdot5^{2g+9}\big).
	\end{align*}
\end{ex}

\begin{ex}
	We provide an example of the $l=l(\mu)=2$ case.
	When $\mu=(5,2)$,
	we have
	\begin{align*}
		\sum_{g=0}^\infty 
		\frac{\hbar^{2g+7}}{(2g+7)!} H_{g;(5,2)}
		=&\frac{1}{7!\cdot5\cdot2}
		\cdot\big({\mathrm e}^{21\hbar}
		-6 {\mathrm e}^{14\hbar}-21 {\mathrm e}^{11\hbar}+35 {\mathrm e}^{9\hbar}
		+70 {\mathrm e}^{6\hbar}-84 {\mathrm e}^{4\hbar}-105 {\mathrm e}^{\hbar}\\
		&+105{\mathrm e}^{-\hbar}+84{\mathrm e}^{-4\hbar}-70{\mathrm e}^{-6\hbar}
		-35{\mathrm e}^{-9\hbar}+21{\mathrm e}^{-11\hbar}+6{\mathrm e}^{-14\hbar}
		-{\mathrm e}^{-21\hbar}\big),
	\end{align*}
	and
	\begin{align*}
		H_{g;(5,2)}
		=&\frac{2}{7!\cdot5\cdot2}
		\cdot\big(21^{2g+7}-6\cdot14^{2g+7}
		-21\cdot11^{2g+7}+35\cdot9^{2g+7}
		+70\cdot 6^{2g+7}\\
		&-84\cdot4^{2g+7}
		-105\cdot 1^{2g+7}.
	\end{align*}
\end{ex}

\begin{ex}
	We provide an example of the $l=l(\mu)=3$ case.
	When $\mu=(3,2,1)$,
	we have
	\begin{align*}
		\sum_{g=0}^\infty 
		\frac{\hbar^{2g+7}}{(2g+7)!} H_{g;(3,2,1)}
		=&\frac{1}{6!\cdot3\cdot2}
		\cdot\big({\mathrm e}^{15\hbar}
		-6 {\mathrm e}^{10\hbar}
		-15 {\mathrm e}^{7\hbar}
		-20 {\mathrm e}^{6\hbar}
		+39 {\mathrm e}^{5\hbar}
		+120 {\mathrm e}^{4\hbar}
		+35 {\mathrm e}^{3\hbar}\\
		&-150 {\mathrm e}^{2\hbar}
		-210 {\mathrm e}^\hbar
		+210{\mathrm e}^{-\hbar}+150{\mathrm e}^{-2\hbar}
		-35{\mathrm e}^{-3\hbar}
		-120{\mathrm e}^{-4\hbar}
		-39{\mathrm e}^{-5\hbar}\\
		&+20{\mathrm e}^{-6\hbar}
		+15{\mathrm e}^{-7\hbar}
		+6{\mathrm e}^{-10\hbar}
		-{\mathrm e}^{-15\hbar}\big),
	\end{align*}
	and
	\begin{align*}
		H_{g;(3,2,1)}
		=&\frac{2}{6!\cdot3\cdot2}
		\cdot\big(15^{2g+7}-6 \cdot10^{2g+7}
		-15 \cdot7^{2g+7}-20 \cdot6^{2g+7}
		+39 \cdot5^{2g+7}\\
		&+120 \cdot4^{2g+7}
		+35 \cdot3^{2g+7}-150 \cdot2^{2g+7}
		-210 \cdot1^{2g+7}\big).
	\end{align*}
\end{ex}

\section{Conflict of interest and data availability statement}
The author states that there is no conflict of interest, and
no datasets were generated or analysed during the current study.

\vspace{.2in}
{\em Acknowledgements}.
The author would like to thank Xiaobo Liu, Jian Zhou, and Xiangyu Zhou for their encouragement and Xuhui Zhang for helpful discussions.
The author is supported by the NSFC grants (No. 12288201, 12401079),
the China Postdoctoral Science Foundation (No. 2023M743717),
and the China National Postdoctoral Program for Innovative Talents (No. BX20240407).

\vspace{.2in}

\renewcommand{\refname}{Reference}
\bibliographystyle{plain}
\bibliography{reference}
\vspace{30pt} \noindent
\end{document}